\documentclass[11pt]{amsart}
\usepackage[english]{babel}
\usepackage[T1]{fontenc}
\usepackage[utf8]{inputenc}

\makeatletter
\@namedef{subjclassname@2020}{\textup{2020} Mathematics Subject Classification}
\makeatother

\usepackage{graphicx}
\usepackage{amssymb,amscd,amsthm,amsxtra}
\usepackage{latexsym}
\usepackage{epsfig}

\usepackage{fancyhdr}
\usepackage{pgf} 

\usepackage{mathtools}
\usepackage{esint}
\usepackage[active]{srcltx}
\usepackage{color}
\usepackage{hyperref}
\hypersetup{colorlinks,breaklinks,
            linkcolor=[rgb]{0,0,0},
            citecolor=[rgb]{0,0,0},
            urlcolor=[rgb]{0,0,0}}

\numberwithin{equation}{section}
\newcommand{\be}{\begin{equation}}
\newcommand{\ee}{\end{equation}}

\newcommand{\R}{\mathbb R}
\newcommand{\N}{\mathbb N}

\newcommand{\eps}{\varepsilon}
\newcommand{\loc}{\mathrm{loc}}

\newcommand{\comment}[1]{}

\newcommand{\abs}[1]{\lvert {#1} \rvert}
\newcommand{\norm}[2]{\Vert {#1} \Vert_{#2}}

\newtheorem{teo}{Theorem}[section]

\newtheorem{lemma}[teo]{Lemma}
\newtheorem{theorem}[teo]{Theorem}
\newtheorem{proposition}[teo]{Proposition}
\theoremstyle{definition}
\newtheorem{definition}[teo]{Definition}

\newtheorem{remark}[teo]{Remark}

\newcommand{\cV}{\mathcal{V}}
\newcommand{\pa}{\partial}
\newcommand{\de}{\mathrm{d}}

\usepackage{cfr-lm}
\usepackage[osf]{libertine}

\begin{document}

\title[On the nodal set of solutions to some sublinear equations  without homogeneity]{On the nodal set of solutions to some sublinear equations\\ without homogeneity}
\thanks{
\indent \emph{Acknowledgements:} The authors are partially supported by the INDAM-GNAMPA group. G. T. is partially supported by the ERC project no.\ 853404 \emph{Variational approach to the regularity of the free
boundaries - VAREG} held by Bozhidar Velichkov. \\
Part of this work was carried out while N. S. was visiting the University of Pisa, which he wish to thank for the hospitality.\\
We thank the anonymous referees for the careful reading of the manuscript, and for precious suggestions.}

\author[N. Soave]{Nicola Soave}\thanks{}
\address {Nicola Soave \newline \indent
	Dipartimento di Matematica, Universit\`a degli Studi di Torino\newline \indent
	Via Carlo Alberto 10, 10123 Torino, Italy}
\email{nicola.soave@unito.it}

\author[G. Tortone]{Giorgio Tortone}\thanks{}
\address {Giorgio Tortone \newline \indent
	Dipartimento di Matematica, Universit\`a di Pisa \newline \indent
	Largo B. Pontecorvo 5, 56127 Pisa - ITALY}
\email{giorgio.tortone@dm.unipi.it}

\begin{abstract}
We investigate the structure of the nodal set of solutions to an unstable Alt-Phillips type problem
\[
-\Delta u = \lambda_+(u^+)^{p-1}-\lambda_-(u^-)^{q-1}
\]
where $1 \le p<q<2$, $\lambda_+ >0$, $\lambda_- \ge 0$. The equation is characterized by the sublinear \emph{inhomogeneous} character of the right hand-side, which makes it difficult to adapt in a standard way classical tools from free-boundary problems, such as monotonicity formulas and blow-up arguments. Our main results are: the local behavior of solutions close to the nodal set; the complete classification of the admissible vanishing orders, and estimates on the Hausdorff dimension of the singular set, for local minimizers; the existence of degenerate (not locally minimal) solutions.
\end{abstract}

\maketitle

\medskip

{\small \noindent \text{Statements and Declarations:} The authors have no relevant financial or non-financial interests to disclose.}

\smallskip

{\small \noindent \text{Data availability:} Data sharing not applicable to this article as no datasets were generated or analysed during the current study.}

\section{Introduction}
The purpose of this paper is to study the structure of
the nodal sets, and the local behavior nearby, of solutions to the following equations with sublinear nonlinearities, posed in a domain $\Omega \subset \R^n$, $n \ge 2$:
\be\label{eq0}
-\Delta u = \lambda_+(u^+)^{p-1}-\lambda_-(u^-)^{q-1}, \quad \text{where $\lambda_{\pm}>0$ and $1 \le p<q < 2$},
\ee
and
\be\label{equation1}
-\Delta u = \lambda_+(u^+)^{p-1},\quad \text{where $\lambda_{+}>0$ and $1 \le p<2$}.
\ee
Both problems are characterized by the asymmetric behavior of the positive and the negative parts, defined as usual as $u^+:= \max\{u,0\}$ and $u^-:= \max\{-u,0\}$. In particular, the inhomogeneity of the right hand side in equation \eqref{eq0} prevents the existence of homogeneous solutions, and destroys any natural scaling in the equation. This challenging feature arises in different applications related to the composite membrane problem (see \cite{Bl, MR2421158, Chanillo99thefree,MR1796024,  MR2283955} ) as well as in the case of solid combustion problems with ignition temperature (see \cite{combu3,combu1,combu4, combu2}). Since tools from free-boundary problems, such as monotonicity formulas and blow-up arguments, are usually obtained by exploiting the existence of such natural scaling, our analysis will need several new ingredients.


The motivation for our study comes from some recent results regarding the structure of the nodal sets of solutions to elliptic sublinear equations. It is well known that solutions to the (super)linear problem
\be\label{eq intro}
-\Delta u = f(x,u) \quad \text{in $\Omega$},
\ee
where $f$ is a Caratheodory function satisfying
\begin{equation}\label{superlinear}
|f(x,t)| \le C |t| \qquad \forall (x,t) \in \Omega \times [-M,M], \quad \forall M>0,
\end{equation}
obey the strong unique continuation principle (see e.g. \cite{AKS, GL}); in particular, $u=0$ in an open subset $\omega \subset \Omega$ implies that $u \equiv 0$ in the whole $\Omega$. Moreover, it was proved in \cite{CafFri85} that any solution $u$ behaves like a homogeneous harmonic polynomial close to each point of its nodal set (see also \cite{Ber}). As a consequence, one can shows that the Hausdorff dimension of its singular set $S(u) = \{u=0=|\nabla u|\}$ is at most $n-2$, and the singular set is actually discrete when $n=2$ \cite{CafFri85} (the regular part $R(u) = \{u=0, \ \nabla u \neq 0\}$ is obviously a smooth $(n-1)$-dimensional hypersurface, by the implicit function theorem). The result was further generalized and refined in \cite{Han94}, where it is proved the full stratification of the singular set, even for more general equations in divergence form. Measure estimates on the nodal and singular set were investigated in \cite{DoFe, HHL, Lin91} and in the more recent papers \cite{CNV, logu2} (see also the references therein).

If the linear growth condition \eqref{superlinear} is replaced by a sublinear one, namely
\begin{equation}\label{hp sub}
|f(x,t)| \le C |t|^\alpha \qquad \forall (x,t) \in \Omega \times [-M,M], \quad \forall M>0,
\end{equation}
for some $\alpha \in [0,1)$, then even the unique continuation principle fails in general, as it is not difficult to construct dead-core solutions in dimension $n=1$ (and hence in any dimension): for example, if $p \in (1,2)$, then $u(x) = c_p (t^+)^{\frac{2}{2-p}}$ solves $u'' = |u|^{p-2}u$ on $\R$ for an appropriate choice of the constant $c_p$, and has arbitrarily large nodal (actually singular) set.

However, it was recently observed in \cite{soaveweth} that, if in \eqref{eq intro} one imposes \eqref{hp sub} and the additional sign-assumption
\[
0<f(x,s)s \quad \text{for all $s \in (-\eps,\eps) \setminus
 \{0\}$}
\]
(for some small $\eps$), then the unique continuation principle holds; we refer to \cite[Theorem 1.2]{soaveweth} for the precise statement (the proofs in \cite{soaveweth} make use of Almgren-type monotonicity formulas; we also refer to \cite{Rul} for an alternative approach based on Carleman estimates). As a particular case, solutions to \eqref{eq0} with $p ,q\in [1,2)$ cannot vanish on an open subset of $\Omega$, unless $u \equiv 0$.

The validity of the unique continuation principle can be considered as the first step in the study of the geometric structure of the nodal set, and of the local behavior nearby. In \cite{soavesublinear}, it is provided the full description for solutions to
\be\label{eq.sublinear.ST}
-\Delta u = \lambda_+(u^+)^{p-1} -\lambda_-(u^-)^{p-1} \quad \text{in $\Omega$, where $\lambda_{\pm}>0$, and $p \in [1,2)$}.
\ee
More precisely, the authors characterized the admissible vanishing orders of the solutions, described the local behavior near the nodal set through a blow-up analysis, derived optimal Hausdorff dimension estimates on the singular set, and proved the existence and multiplicity of infinitely many homogeneous solutions with the same degree $2/(2-p)$ (this implies in particular the non-validity of measure estimates similar to those obtained in the linear setting).

The analysis in \cite{soavesublinear} crucially relies on the homogeneity of the problem, and on the fact that both $u^+$ and $u^-$ follow the same behavior (namely $p=q$ and $\lambda_\pm$ are both positive). It is natural to wonder what happens if one of these assumptions is removed, destroying the homogeneity and introducing an asymmetry in the behavior of the positive and negative parts. This consists precisely in considering \eqref{eq0} or \eqref{equation1}. It will emerge that the behavior of solutions to these equations can be rather different with respect to those of \eqref{eq.sublinear.ST}. We also anticipate that equations \eqref{eq0} and \eqref{equation1} are strictly connected, since homogeneous solutions to \eqref{equation1} will appear as blow-up limits of solutions to \eqref{eq0} at some nodal points.

\medskip

Before proceeding with the formal statement of our main results, we conclude this introduction by mentioning other results related to ours.

Concerning \eqref{equation1}, we observe that it can be considered as a generalization of problem
\be\label{unst}
-\Delta u= \lambda_+ \chi_{\{u>0\}} \quad \text{in $\Omega$, where $\lambda_{+}>0$},
\ee
which is known in the literature as \emph{the unstable obstacle problem}, studied in \cite{ASW1, ASW, AW, MW}. This corresponds to the case $p=1$ in \eqref{equation1}. The word \emph{unstable} comes from the fact that solutions to \eqref{unst} may not be locally minimal (i.e. stable), and this creates a number of interesting new features with respect to the classical (stable) obstacle problem, when the $-$ sign in front of the Laplacian is replaced by $+$: for instance, \eqref{unst} admits solutions which are degenerate of second order, or solutions of class $C^{1,\alpha}$ for every $\alpha \in (0,1)$, but not of class $C^{1,1}$, see \cite{AW}. Our results for \eqref{equation1} generalize part of the analysis in \cite{AW, MW} in the full range $p \in [1,2)$, with some remarkable differences both in the statements, and in the proofs. Indeed, solutions to \eqref{unst} are harmonic when negative, while solve the Poisson equation $-\Delta u =\lambda_+$ when positive. This allows to construct more or less explicit barriers, which are crucially employed to prove, for instance, the non-degeneracy of minimal solutions \cite[Section 3]{MW}, and the classification of the homogeneous ones \cite[Remark 3.3]{ASW}. The strategy to treat these issues for general $p \in (1,2)$ will be completely different.

\medskip

In light of the connection between stable and unstable problems, our results can be seen as an \emph{unstable Alt-Phillips type problem with inhomogeneity}. Starting from the seminal papers \cite{AP,P}, different groups of authors addressed, in the stable setting, the minimization problem associated to functionals of type
\[
\int_{\Omega} \left( \frac12 |\nabla u|^2 + \lambda_+ (u^+)^p + \lambda_-(u^-)^p\right) \mathrm{d}x, \quad \text{with} \quad p \in (0,2), \lambda_{\pm} \ge 0.
\]
Despite several contributions to the theory, some questions related to the singular set are still open, highlighting how the presence of a semilinear term complicates the analysis of the free boundary substantially: see \cite{FS} for the case $p \in (1,2)$, and \cite{FerYu, LiPe} for $p \in (0,1)$. As far as the unstable case is concerned (which corresponds to considering $-$ instead of $+$ in front of $\lambda_{\pm}$ in the functional), in addition to \cite{soavesublinear} ($p=q \in [1,2)$) we also mention \cite{soavesingular}, which concerns the \emph{unstable singular case} $p =q \in(0,1)$. Moreover, analogue \emph{non-local} problems have been recently studied, see e.g. \cite{Tor}. Both from the stable and the unstable case, this seems to be the first contribution where positive and negative parts have different powers.

\medskip

Finally, regarding generalizations to the unique continuation properties in \cite{Rul, soaveweth}, we mention that the first result in this direction was obtained in \cite{PaWe} for \emph{least energy solutions} of a Neumann boundary value problem; the proof in \cite{PaWe} strongly exploits the minimality. More recently, various results have been obtained for \emph{parabolic sublinear equations} \cite{ArBa, BaMa}, and for a class of \emph{degenerate sublinear equations} \cite{BaGaMa}.

\subsection*{Main results and organization of the paper.} At first we show the existence of non-trivial weak solutions to \eqref{eq0} and \eqref{equation1} as minimizers of an associated functional in a space of functions with fixed traces. Let
\[
J(u, \Omega):= \int_{\Omega} \left( \frac{1}{2} |\nabla u|^2- \frac{\lambda_+}{p} (u^+)^p - \frac{\lambda_-}{q}(u^-)^q\right)\de x
\]
and, for $g \in H^1(\Omega)$, let
\[
\mathcal{A}:= \left\{u \in H^1(\Omega): \ u-g \in H^1_0(\Omega)\right\}.
\]

\begin{theorem}\label{thm: ex}
Let $\Omega$ be a bounded regular domain of $\mathbb{R}^n$, let $\lambda_{+}>0$, $\lambda_- >0 $ (resp. $\lambda_-=0$), and suppose that $1 \le p<q <2$. Then there exists a solution to \eqref{eq0} (resp. \eqref{equation1}) obtained as minimizer of $J(\cdot\,,\Omega)$ in $\mathcal{A}$.
\end{theorem}

Once that the existence of weak solutions is established, we analyze their behavior close to the zero level set $\{u=0\}$. Since the following results are of local nature, we suppose without loss of generality that $\Omega=B_1$ where, as usual, we let $B_r(x_0)$ denote the ball of center $x_0$ and radius $r$ in $\R^n$, and, in the frequent case $x_0=0$, we often write $B_r$ instead of $B_r(0)$, for the sake of brevity. By standard regularity theory, any weak solution to \eqref{eq0} or \eqref{equation1} is of class $C^{1,\alpha}(\Omega)$ for every $\alpha \in (0,1)$, and even of class $C^2(\Omega)$ if $p>1$. Therefore, it makes sense to consider the point-wise value of $u$ and of its gradient, and to split the nodal set into the regular part $R(u)=\{u=0, |\nabla u| \neq 0\}$, which is a smooth $(n-1)$-dimensional hypersurface, and the singular set $S(u)=\{u=0=|\nabla u|\}$. We introduce now the notion of vanishing order.

\begin{definition}\label{def: order V}
Let $u \in H^1_{\loc}(\Omega) \cap C^{1,\alpha}(\Omega)$, and let $x_0 \in \{u=0\}$. The \emph{vanishing order of $u$ at $x_0$} is defined as
\[
\cV(u,x_0) = \sup\left\{\beta>0: \ \limsup_{r \to 0^+} \frac1{r^{n-1+2\beta}} \int_{\pa B_r(x_0)} u^2\,\de\sigma<+\infty\right\}.
\]
\end{definition}

The number $\cV(u,x_0) \in \R^+ \cup \{+\infty\}$ is characterized by the property that
\[
\limsup_{r \to 0^+} \frac1{r^{n-1+2\beta}} \int_{\pa B_r(x_0)} u^2\, \de \sigma  = \begin{cases} 0 & \text{if $0 <\beta< \cV(u,x_0)$} \\
+\infty  & \text{if $\beta > \cV(u,x_0)$}.
\end{cases}
\]

In terms of the vanishing order, the strong unique continuation property (SUCP) consists in the fact that non-trivial solutions cannot vanish with infinite order at any nodal point. This implies in particular that non-trivial solutions cannot vanish identically in any open subset of their reference domain, which is the classical weak unique continuation principle. The validity of the SUCP for solutions to both \eqref{eq0} and \eqref{equation1} is essentially known from \cite{Rul, soavesublinear} (apart from some particular cases, discussed in Proposition \ref{prop: UCP} below) and it implies the existence of a finite vanishing order at every nodal point.

In the special case of equation \eqref{eq.sublinear.ST}, it is proved in \cite{soavesublinear} a much stronger fact: denoting by
\be\label{gamma_p}
\gamma_p := \frac{2}{2-p},
\ee
the critical exponent associated to $p$, and by $\beta_p \in \N$ the maximal positive integer strictly smaller than $\gamma_p$, that is
\be\label{betap}
\beta_p :=
\begin{cases}
  \lfloor\gamma_p\rfloor, & \mbox{if } \gamma_p \not\in \N \vspace{0.1cm}\\
  \gamma_p-1 & \mbox{if } \gamma_p \in \N
\end{cases}
\ee
(where $\lfloor \cdot \rfloor$ is the integer part), then
\begin{equation}\label{class V ST}
\text{$u$ solves \eqref{eq.sublinear.ST} and $x_0 \in \{u=0\}$} \quad \implies \quad \cV(u,x_0) \in \{k \in \mathbb{N}: \ k \le \beta_p\} \cup \{\gamma_p\},
\ee
see \cite[Theorem 1.3]{soavesublinear}. That is, for solutions to \eqref{eq.sublinear.ST} there are only finitely many admissible vanishing orders, with a universal bound depeding only on $p$. This is in contrast to what happens for solutions to linear or superlinear elliptic problems. For instance, the Laplace equation has solutions with any (arbitrarily large) integer vanishing order.

By carefully looking at the proof of \eqref{class V ST}, it emerges that both the fact that $\lambda_->0$ and the symmetry condition $q=p$ play important roles. Therefore, we wish to understand whether a similar property holds for solutions to \eqref{eq0} and \eqref{equation1} or not. In this generality, this remains an open problem.
However, we can provide a full classification for a specific class of solutions, namely \emph{local minimizers}.

\begin{definition}
We say that $u \in H^1_{\loc}(\Omega)$ is a \emph{local minimizer} of \eqref{eq0} or \eqref{equation1} if there exists $\eps>0$ such that
\[
J(u, \Omega) \le J(u+v,\Omega) \qquad \forall v \in H_0^1(\Omega) \ \text{with} \ \|v\|_{H^1(\Omega)} \le \eps.
\]
\end{definition}

Plainly, the solutions found in Theorem \ref{thm: ex} are local minimizers.

\begin{theorem}\label{thm: van ord}
Let $u \in H^1_{\loc}(\Omega)$ be a non-trivial local minimizer of equation \eqref{eq0} or \eqref{equation1}, and let $x_0 \in \{u=0\}$. Then
\[
\cV(u,x_0) \in \{k \in \mathbb{N}: \ k \le \beta_p\} \cup \{\gamma_p\}.
\]
\end{theorem}

It is remarkable that, despite the presence of two different exponents $p$ and $q$ in equation \eqref{eq0}, and hence also of two different critical exponents $\gamma_p$ and $\gamma_q$ (with $\gamma_q \to +\infty$ as $q \to 2^-$), we only see the smaller one $\gamma_p$ in the classification of the admissible vanishing orders. Our proof exploits the local minimality in a crucial way, and proceeds as follows: with a preliminary study, valid for any solution to \eqref{eq0} or \eqref{equation1} (not necessarily locally minimal), we show that
\[
\mathcal{V}(u,x_0) \in \{k \in \mathbb{N}: \ k \le \beta_p\} \cup [\gamma_p, +\infty), \qquad \forall x_0 \in \{u=0\}.
\]
Moreover, if $\mathcal{V}(u,x_0) \in \{1,\dots,\beta_p\}$, then we have an expansion of type
\[
u(x) =  P_{x_0}(x-x_0) + \Gamma_{x_0}(x),
\]
where $P_{x_0}$ is a non-trivial homogeneous harmonic polynomial of degree $\mathcal{V}(u,x_0)$, and $\Gamma_{x_0}$ is a higher-order remainder which can be conveniently controlled (see Proposition \ref{prop.underdog}). This intermediate result implies that we can focus on those $x_0 \in \{u=0\}$ with vanishing order $\mathcal{V}(u,x_0) \ge \gamma_p$ and, to this purpose, we introduce the following definition.

\begin{definition}\label{def.nondeg}
Let  $u\in H^1_\loc(B_1)$ be a non-trivial solution of \eqref{eq0} or \eqref{equation1}, and let $x_0 \in \{u=0\}$ such that $\mathcal{V}(u,x_0)\geq \gamma_p$. We say that $u$ is $\gamma_p$-non-degenerate at $x_0$ if
$$
\liminf_{r \to 0^+} \frac{1}{r^{\gamma_p}}\sup_{B_r(x_0)}|u| >0,
$$
and it is $\gamma_p$-degenerate if the $\liminf$ is $0$. Moreover, we say that $u$ is a $\gamma_p$-non-degenerate solution if it is $\gamma_p$-non-degenerate at every nodal point such that $\mathcal{V}(u,x_0)\geq \gamma_p$.
\end{definition}

As a second step in the proof of Theorem \ref{thm: van ord}, we show in Proposition \ref{thm: non-deg} below that any local minimizer is $\gamma_p$-non-degenerate.

Finally, Theorem \ref{thm: van ord} will be obtained by combining the $\gamma_p$-non-degeneracy with the following blow-up alternative (which is valid for every non-trivial solution, not necessarily locally minimal), based on the study of the monotonicity and oscillation properties of the \emph{Weiss-type functional}
\begin{align*}
W_{\gamma_p,2}(u,x_0,r) =&  \frac{1}{r^{n-2+2\gamma_p}}\int_{B_r(x_0)} \left( \abs{\nabla u}^2- 2 \left[\frac{\lambda_+}{p} (u^+)^p + \frac{\lambda_-}{q}(u^-)^q\right]\right)\mathrm{d}x\\
& \qquad - \frac{\gamma_p}{r^{n-1+2\gamma_p}} \int_{\partial B_r(x_0)}{u^2 \, \mathrm{d}\sigma}.
\end{align*}


\begin{theorem}\label{thm.mainblow}
  Let $u\in H^1_\loc(B_1)$ be a non-trivial solution of either \eqref{eq0}, or \eqref{equation1}, with $x_0 \in \{u=0\}$. If the vanishing order $\mathcal{V}(u,x_0) \ge \gamma_p=2/(2-p)$, then the following alternative holds:
 \begin{enumerate}
  \item Either
  \be\label{hp blow-up reg}
  \limsup_{r \to 0^+}\frac{1}{r^{n-1+2\gamma_p}}\int_{\partial B_r(x_0)}u^2\,\mathrm{d}\sigma < +\infty;
  \ee
  then $W_{\gamma_p,2}(u,x_0,0^+)>-\infty$ and, for every sequence $r \to 0^+$, there exists a subsequence $r_k \searrow 0^+$ such that
  $$
  \frac{u(x_0+r_k x)}{r_k^{\gamma_p}}
  \to \overline{u} \quad\mbox{in } C^{1,\alpha}_\loc(\R^n),
  $$
  for every $\alpha \in (0,1)$, where $\overline{u}$ is a $\gamma_p$-homogeneous solution to \eqref{equation1}, that is
  \[
  -\Delta \overline{u} = \lambda_+(\overline{u}^+)^{p-1} \quad\mbox{in }\R^n.
  \]
  Moreover
  \[
  W_{\gamma_p,2}(u,x_0,0^+) \geq 0 \quad \iff \quad \bar u \equiv 0 \quad \iff \quad \text{$u$ is $\gamma_p$-degenerate at $x_0$}.
  \]
  \item Or  \be\label{case2}
\limsup_{r \to 0^+}\frac{1}{r^{n-1+2\gamma_p}}\int_{\partial B_r(x_0)}u^2\,\mathrm{d}\sigma = +\infty;
    \ee
then $W_{\gamma_p,2}(u,x_0,0^+)< 0$ (possibly $-\infty$) and $\mathcal{V}(u,x_0) = \gamma_p$. Moreover, there exists a subsequence $r_k \searrow 0^+$ such that
$$
\frac{u(x_0+r_k x)}{\left(\frac{1}{r_k^{n-1}}\int_{\partial B_{r_k}(x_0)}u^2\,\mathrm{d}\sigma\right)^{1/2}}
\to \widetilde u\quad\mbox{in } C^{1,\alpha}_\loc(\R^n),
$$
for every $\alpha \in (0,1)$, where $\widetilde u$ is a $\gamma_p$-homogeneous non-trivial harmonic polynomial. This alternative is possible only if $\gamma_p \in \N$.
\end{enumerate}
\end{theorem}

Roughly speaking, Theorem \ref{thm.mainblow} says that, if the vanishing order $\mathcal{V}(u,x_0) \ge \gamma_p$, then: either $u$ behaves like a $\gamma_p$-homogeneous non-trivial solution of \eqref{equation1} close to $x_0$ (case (i) with $\bar u \not \equiv 0$), and in this case plainly $\mathcal{V}(u,x_0) =\gamma_p$; or $u$ behaves like a $\gamma_p$-homogeneous non-trivial harmonic polymonial close to $x_0$, and again $\mathcal{V}(u,x_0) = \gamma_p$ (this case is possible only if $\gamma_p \in \mathbb{N}$); or else $u$ is $\gamma_p$-degenerate at $x_0$.

Since any local minimizer is $\gamma_p$-non-degenerate (Proposition \ref{thm: non-deg}), Theorem \ref{thm: van ord} follows rather directly.

Statements similar to Theorem \ref{thm.mainblow} already appeared in the literature \cite{Bl, MW} when treating unstable-obstacle-type problems, to study the behavior of singular points with respect to quadratic scalings.

\begin{remark}
It is worth to point out that, for solutions to \eqref{eq0}, where $p<q$, the Weiss-type functional $W_{\gamma_p,2}(u, x_0, \cdot)$ is not necessarily monotone. This fact, related to the inhomogeneity of the problem, is one of the obstruction towards the complete classification of the vanihsing orders for $\gamma_p$-non-degenerate solutions.

We had to introduce a correcting additional term to obtain a monotone quantity. In this way, we could prove that the limit
\[
W_{\gamma_p,2}(u,x_0,0^+) = \lim_{r \to 0^+} W_{\gamma_p,2}(u,x_0,r)
\]
does exist (not only up to a subsequence), see Lemma \ref{lem: ex limit} below.

Note that we did not prove the uniqueness of the blow-up limit in general. However, the fact that $W_{\gamma_p,2}(u,x_0,0^+) \ge 0$ implies that \emph{any blow-up limit must vanish}, and hence in this case we do have uniqueness of blow-ups. If $W_{\gamma_p,2}(u,x_0,0^+) < 0$, we can instead conclude that \emph{any blow-up limit is non-trivial}.
\end{remark}


The blow-up alternative is not only useful in the proof of Theorem \ref{thm: van ord}, but ensures also the existence of \emph{non-trivial homogeneous blow-up limits} at any nodal point, for $\gamma_p$-non-degenerate solutions. It is interesting that such blow-up limits only see the smaller power $p$, since they either solve \eqref{equation1}, or are $\gamma_p$-homogeneous harmonic functions, even if we start from a solution of \eqref{eq0}.

At this point it is natural to investigate existence and properties of $\gamma_p$-homogeneous solutions to \eqref{equation1}. This is the content of Section \ref{section.two}, where we obtain a complete classification in dimension $n=2$ for any $p \in (1,2)$. The case $p=1$ was previously treated in \cite[Remark 3.3]{ASW}.

\begin{theorem}\label{thm: homog}
Let $p \in (1,2)$. Then the number of the homogeneous solutions, modulo rotations, of equation \eqref{equation1} in $\R^2$ is the number of the positive integers in the interval $(\gamma_p,2\gamma_p)$. More precisely, for any $k \in (\gamma_p, 2\gamma_p) \cap \mathbb{N}$, there exists precisely one $\gamma_p$-homogeneous solution with $2k$ zeros on the unit sphere $\partial B_1$, up to rotations.
\end{theorem}

This is very different to what happens both for linear equations (e.g. the Laplace equation has infinitely many homogeneous solutions with arbitrary integer degree), and for the sublinear equation \eqref{eq.sublinear.ST} (which admits infinitely many $\gamma_p$-homogeneous solutions, see \cite[Theorem 1.10]{soavesublinear}). Instead, it is a non-trivial generalization of what happens for the unstable obstacle problem \eqref{unst}, which admits a unique $2$-homogeneous solution, modulo rotations; as shown in \cite[Remark 3.3]{ASW}, this solution has exactly $6$ zeros on the unit circle. Now notice that \eqref{unst} is precisely \eqref{equation1} with $p=1$; in such case, $\gamma_p=2$, and there exists precisely one integer $k=3$ in the interval $(\gamma_p,2\gamma_p) = (2,4)$.

Once that homogeneous solutions in the plane are classified, it is rather standard to infer an optimal bound on the Hausdorff dimension of the singular set, via the dimension reduction principle.

\begin{theorem}\label{thm: hausdorff}
Let $u$ be either a solution to \eqref{eq0}, or a solution to \eqref{equation1}, in $\Omega \subset \R^n$. Suppose that $u$ is $\gamma_p$-non-degenerate. Then the Hausdorff dimension of the singular set $S(u)= \Omega \cap \{u=0=|\nabla u|\}$ is less than or equal to $n-2$. Moreover, the singular set is discrete in dimension $n=2$.

In particular, the thesis holds for local minimizers of \eqref{eq0} or \eqref{equation1}.
\end{theorem}

Theorems \ref{thm: van ord}-\ref{thm: hausdorff} provide a complete description for local minimizers, and, more in general, for $\gamma_p$-non-degenerate solutions.

At this point it is natural to investigate the existence of $\gamma_p$-degenerate solutions. An adaptation of the strategy in \cite[Corollary 4.4]{AW} (which concerns the case $p=1$, $\lambda_-=0$) allows us to prove the existence of $\gamma_p$-degenerate solutions for both \eqref{eq0} and \eqref{equation1}.

\begin{theorem}\label{thm: ex deg}
Let $\lambda_+>0$, $\lambda_->0$ (resp. $\lambda_-=0$), and $1<p<q<2$. There exists a non-trivial solution $u$ of \eqref{eq0} (resp. \eqref{equation1}) in $B_1$ which is $\gamma_p$-degenerate at the origin, that is
\[
\liminf_{r \to 0^+} \frac1{r^{\gamma_p}} \sup_{B_r(0)} |u| = 0.
\]
\end{theorem}

The classification of the vanishing orders and the blow-up analysis for these degenerate solutions remain interesting open problems.

Note that the assumption $p<q$ is crucial in Theorem \ref{thm: ex deg}, since in \cite{soavesublinear} it is showed that, when $1 \le p=q<2$, then \emph{all} solutions are $\gamma_p$-non-degenerate. This difference enters also in the blow-up analysis; thus, from this point of view, our analysis of \eqref{eq0} is somehow closer to the one carried out in \cite{MW} for the unstable obstacle problem, rather than to the one in \cite{soavesublinear}.

\begin{remark}
Equations \eqref{eq0} and \eqref{equation1} could be also considered in the singular range, when $p \in (0,1)$ (and $q \in (p,2)$). While our approach exploits the fact that $p \ge 1$ in several steps, it is natural to expect that for $p \in (0,1)$ a singular perturbation argument may lead to existence of solutions, and may allow to study the blow-up behavior. This is left as an open problem. See for instance \cite{soavesingular} for a similar approach in the homogeneous case $0<p=q<1$.
\end{remark}

\subsection*{Structure of the paper.} In Section \ref{section.preli}, we present some results which are essentially known, or can be obtained by previous contributions with minor changes, and which will be frequently used in the rest of the paper. Section \ref{sec: ex and non-deg} is devoted to the existence of minimizers for the problem with fixed traces (Theorem \ref{thm: ex}), and to the $\gamma_p$-non-degeneracy of local minimizers (Proposition \ref{thm: non-deg}). The blow-up analysis is the content of Section \ref{section.blowup}, which contains the proofs of Theorems \ref{thm: van ord} and \ref{thm.mainblow}. Section \ref{section.two} is devoted to the classification of homogeneous solutions to \eqref{equation1} (Theorem \ref{thm: homog}), and to the estimate on the Hausdorff dimension of the singular set for $\gamma_p$-non-degenerate solutions (Theorem \ref{thm: hausdorff}). Finally, Section \ref{section.twodeg} contains the construction of the $\gamma_p$-degenerate solutions (Theorem \ref{thm: ex deg}).

\section{Preliminaries}\label{section.preli}

This section is devoted to some results which will be frequently used throughout the rest of the paper.

At first, we present a preliminary analysis on the local behavior of solutions to \eqref{eq0} or \eqref{equation1} close to the nodal set. Recall that $\mathcal{V}(u,x_0)$ denotes the vanishing order, introduced in Definition \ref{def: order V}, and that $\gamma_p$ is defined in \eqref{gamma_p}.

\begin{proposition}\label{prop.underdog}
Suppose that $u \in H^1_\loc(B_1)$ is a non-trivial solution to \eqref{eq0} or \eqref{equation1}. Then, for any $x_0 \in \{u=0\}$ and $r \in (0,1-|x_0|)$, the following alternative holds:
  \begin{enumerate}
    \item either $\mathcal{V}(u,x_0) \in \{k \in \N\setminus \{0\}\colon k\leq \beta_p\}$ and there exist a non-trivial homogeneous harmonic polynomial $P_{x_0}$ of degree $\mathcal{V}(u,x_0)$ and a function $\Gamma_{x_0}$ such that
        \be\label{taylor}
        u(x) = P_{x_0}(x-x_0) + \Gamma_{x_0}(x) \quad \mbox{in }B_r(x_0)
        \ee
        with
        $$
        \begin{cases}
        |\Gamma_{x_0}(x)| \leq C|x-x_0|^{\mathcal{V}(u,x_0)+\delta}\\
        |\nabla\Gamma_{x_0}(x)| \leq C|x-x_0|^{\mathcal{V}(u,x_0)-1+\delta}
        \end{cases}\quad\mbox{in }B_r(x_0)
        $$
        for some constants $C,\delta>0$;
    \item or $\mathcal{V}(u,x_0)\geq \gamma_p$ and for every $\eps>0$ there exists $C_\eps>0$ such that $$
        \begin{cases}
        |u(x)| \leq C_\eps|x-x_0|^{\gamma_p-\eps}\\
        |\nabla u(x)| \leq C_\eps|x-x_0|^{\gamma_p -1-\eps}
        \end{cases}\quad\mbox{in }B_r(x_0).
        $$
  \end{enumerate}
\end{proposition}

\begin{proof}
The proposition is stated and proved in \cite[Proposition 2.1]{soavesublinear} for equations \eqref{equation1} and  \eqref{eq.sublinear.ST}. The proof can be repeated almost verbatim also if $u$ solves \eqref{eq0}; it is sufficient to observe that, if $|u(x)| \le C |x-x_0|^\alpha$ for some $\alpha>0$, then
\[
|\Delta u(x)| \le C\lambda_+ |x-x_0|^{(p-1)\alpha} +  C\lambda_- |x-x_0|^{(q-1)\alpha} \le C\max\{\lambda_+,\lambda_-\} |x-x_0|^{(p-1)\alpha},
\]
whenever $|x-x_0| <1$. Thus, the argument used in the proof only ``sees" the smaller power $p<q$.
\end{proof}

Now, following \cite{soavesublinear}, we introduce a family of Weiss-type functionals depending on two parameters.

Precisely, let $u \in H^1_{\loc}(B_1)$ be a solution to \eqref{eq0} or \eqref{equation1} and $x_0 \in \{u=0\}, r \in (0,1-|x_0|)$. For $t>0$, we consider the functionals
\begin{align}\label{E.H}
\begin{aligned}
D_{t}(u,x_0,r)&=\int_{B_r(x_0)} \left( \abs{\nabla u}^2- t F_{\lambda_+,\lambda_-}(u)\right)\mathrm{d}x,\\
H(u,x_0,r)&=\int_{\partial B_r(x_0)}{u^2 \, \mathrm{d}\sigma},
\end{aligned}
\end{align}
where
\[
F_{\lambda_+,\lambda_-}(u) =F_{\lambda_\pm}(u) = \frac{\lambda_+}{p}(u^+)^p + \frac{\lambda_-}{q}(u^-)^q
\]
(note that the definition of $F_{\lambda_\pm}$ slightly differs from the one in \cite{soavesublinear}; as a result, some of the next formulas present minor changes). Moreover, we consider the associated $2$-parameters Weiss-type functionals
\be\label{equation.weiss}
W_{\gamma,t}(u,x_0,r) = \frac{1}{r^{n-2+2\gamma}}D_{t}(u,x_0,r) - \frac{\gamma}{r^{n-1+2\gamma}} H(u,x_0,r)
\ee

\begin{remark}\label{rem: su W scaled}
We stress that both the definition of $D_t$ and of $W_{\gamma,t}$, depends also on the coefficients  $\lambda_+$ and $\lambda_-$ appearing in \eqref{eq0} or \eqref{equation1}. In particular, when we scale $u$, then in general also $\lambda_\pm$ change accordingly. Since however $\lambda_\pm$ are unambiguously determined by $u$, via the differential equation, we will not stress this dependence.
\end{remark}

Now, by proceeding exactly as in \cite[Section 2]{soaveweth} and \cite[Section 2]{soavesublinear}, we can easily obtain the expression of the derivatives of $D_t$, $H$ and $W_{\gamma,t}$ with respect to $r$, for the whole range of parameters $t$ and $\gamma$.

\begin{proposition}\label{weiss.mon}
Let $u \in H^1_{\loc}(B_1)$ solve \eqref{eq0} or \eqref{equation1} in $B_1$, and let $x_0 \in \{u=0\}, r \in (0,1-|x_0|)$. Then
\begin{align*}
\begin{aligned}
  \frac{d}{dr}W_{\gamma,t}(u,x_0,r) =&\, \frac{2}{r^{n-2+2\gamma}}\int_{\partial B_r(x_0)}{\left(\partial_r u -\frac{\gamma}{r}u\right)^2 \mathrm{d}\sigma} + \frac{2-t}{r^{n-2+2\gamma}} \int_{\partial B_r(x_0)}{F_{\lambda_\pm}(u)\,\mathrm{d}\sigma} \\
  &\, - \frac{C_{n,t}-2\gamma(t-p)}{p \, r^{n-1+2\gamma}} \int_{B_r(x_0)} \lambda_+(u^+)^p\, \mathrm{d}x - \frac{C_{n,t}-2\gamma(t-q)}{q \, r^{n-1+2\gamma}}  \int_{B_r(x_0)} \lambda_-(u^-)^q\, \mathrm{d}x,
    \end{aligned}
  \end{align*}
where $C_{n,t} = 2n-t(n-2)$. In particular,
\begin{itemize}
\item[(i)] if $u$ solves \eqref{eq0}, then $r \mapsto W_{\gamma,2}(u,x_0,r)$ is monotone non-decreasing for $\gamma \ge \gamma_q =2/(2-q)$;
\item[(ii)] if $u$ solves \eqref{equation1} then $r \mapsto W_{\gamma,2}(u,x_0,r)$ is monotone non-decreasing for $\gamma \ge \gamma_p =2/(2-p)$, and $W_{\gamma_p,2}(u,x_0,\cdot\,)$ is constant for $r \in (r_1,r_2)$ if and only if $u$ is $\gamma_p$-homogeneous in the annulus $B_{r_2}(x_0) \setminus \overline{B_{r_1}(x_0)}$.
\end{itemize}
\end{proposition}

\begin{proof}
The expression of the derivative of $W_{\gamma,t}$ can be obtained exactly as in \cite[Proposition 2.3]{soavesublinear}. The monotonicity properties in points (i) and (ii) of the thesis follows directly, observing that
\[
C_{n,2}-2\gamma(2-p) \le 0 \quad \iff \quad \gamma \ge \gamma_p, \quad \text{and} \quad C_{n,2}-2\gamma(2-q) \le 0 \quad \iff \quad \gamma \ge \gamma_q.
\]
The case when $u$ solves \eqref{equation1} and $W_{\gamma_p,2}$ is constant in $r$ can be treated as in \cite[Corollary 2.4]{soavesublinear}.
\end{proof}

As highlighted by \cite{soavesublinear}, the monotonicity formulas associated to the parameters $\gamma\ge \gamma_p$, and $t=2$ or $t=p$, play a major role in the local analysis of solution to \eqref{eq.sublinear.ST} close to the nodal points such that $\mathcal{V}(u,x_0)\geq \gamma_p$.

For solutions to \eqref{eq0}, where $p<q$, we only have the monotonicity of $W_{\gamma,2}$ when $\gamma \ge \gamma_q>\gamma_p$. This is not helpful in studying points with vanishing order in the ``intermediate range" $[\gamma_p, \gamma_q]$. This lack of monotonicity makes the blow-up analysis for problem \eqref{eq0} more involved, and is one of the main obstructions towards a complete classification of vanishing orders for general (possibly $\gamma_p$-degenerate) solutions. In such case, at the moment we can only show that the SUCP holds.

\begin{proposition}\label{prop: UCP}
Let $u \in H^1_{\loc}(B_1)$ be a non-trivial solution of \eqref{eq0} or \eqref{equation1} in $B_1$. Then the vanishing order $\mathcal{V}(u,x_0)<+\infty$ for every $x_0 \in \{u=0\}$.
\end{proposition}

\begin{proof}
If $u$ solves \eqref{eq0} with $1<p<q<2$, or \eqref{equation1} in the full range $p \in [1,2)$, then the thesis follows from \cite[Theorem 3]{Rul} or \cite[Theorem 1.2]{soavesublinear}, respectively. It remains to study the case when $u$ solves \eqref{eq0} with $p=1$, which can be treated exactly as in \cite[Theorem 1.2]{soavesublinear}, once that we note that $W_{\gamma,2}(u,x_0,r)$ is monotone non-decreasing whenever $\gamma \ge \gamma_q$, thanks to Proposition \ref{weiss.mon}.
\end{proof}

In the proof of Theorem \ref{thm.mainblow}, when dealing with solutions to \eqref{eq0}, we will exploit the following ``partial monotonicity" result.

\begin{lemma}\label{lem: ex limit}
Let $u \in H^1_{\loc}(B_1)$ be a solution of \eqref{eq0}, $x_0 \in \{u=0\}$, and suppose that $\mathcal{V}(u,x_0) \ge \gamma_p$ (defined by \eqref{gamma_p}).
\begin{enumerate}
\item If $\mathcal{V}(u,x_0) > \gamma_p$, then for every $\eps \in (0,\mathcal{V}(u,x_0)-\gamma_p)$ there exists $C>0$ such that
\[
r \mapsto W_{\gamma_p,2}(u,x_0,r) + C r^{\gamma_p(q-p)+\eps q} \quad \text{is monotone non-decreasing}.
\]
\item If $\mathcal{V}(u,x_0) = \gamma_p$, then for every $\eps \in (0,\gamma_p(1-p/q))$ there exists $C>0$ such that
\[
r \mapsto W_{\gamma_p,2}(u,x_0,r) + C r^{\gamma_p(q-p)-\eps q} \quad \text{is monotone non-decreasing}.
\]
In particular, in both cases there exists the limit
\[
W_{\gamma_p,2}(u,x_0,0^+) = \lim_{r \to 0^+} W_{\gamma_p,2}(u,x_0,r).
\]
\end{enumerate}
\end{lemma}

\begin{proof}
Note that $C_{n,2}- 2\gamma_p(2-p) = 0$, while $C_{n,2}- 2\gamma_p(2-q) = 2\gamma_p(q-p)>0$. Thus, by Proposition \ref{weiss.mon}, we infer that
\be\label{08042}
  \frac{d}{dr}W_{\gamma_p,2}(u,x_0,r) \ge - \frac{2\gamma_p(q-p)}{q \, r^{n-1+2\gamma_p}}  \int_{B_r(x_0)} \lambda_-(u^-)^q\, \mathrm{d}x.
\ee
Now, if $\mathcal{V}(u,x_0)> \gamma_p$, let $\eps \in (0,\mathcal{V}(u,x_0)-\gamma_p)$, and let $\beta=\gamma_p+\eps$. If instead $\mathcal{V}(u,x_0)= \gamma_p$, let $\eps \in (0,\gamma_p(1-p/q))$, and let $\beta=\gamma_p-\eps$. In both cases, being $\beta<\mathcal{V}(u,x_0)$, by definition of vanishing order we have that
\[
H(u,x_0,r) \le C r^{n-1+2\beta} \qquad \forall r <r_0:= 1-|x_0|,
\]
where $C>0$ is a positive constant independent of $r$. Therefore,
\[
\begin{split}
\frac1{r^{n-1+2\gamma_p}}\int_{B_{r}(x_0)} (u^-)^q\,\mathrm{d}x & \le \frac1{r^{n-1+2\gamma_p}}\int_0^r t^{(n-1)\left(1-\frac{q}{2}\right)} H(u,x_0,t)^\frac{q}{2}\,\mathrm{d}t \\
& \le \frac{C}{r^{n-1+2\gamma_p}}\int_0^r t^{n-1+q\beta} \,\mathrm{d}t = C r^{1-2\gamma_p+q\beta},
\end{split}
\]
which, combing back to \eqref{08042}, gives
\[
\frac{d}{dr}W_{\gamma_p,2}(u,x_0,r) \ge - C r^{1-2\gamma_p+q\beta}\quad \forall r <r_0.
\]
The thesis follows since, in both the cases we are considering, we have that $2-2\gamma_p+q\beta>0$ by our choice of $\eps$ and $\beta$ (note also that $2-2\gamma_p+q\beta = \gamma_p(q-p)\pm \eps q$ if $\beta=\gamma_p\pm\eps$).
\end{proof}

\section{Existence and non-degeneracy}\label{sec: ex and non-deg}

We start this section by proving the existence of minimizers for the problem with fixed traces.

\begin{proof}[Proof of Theorem \ref{thm: ex}]
By the Poincar\'e inequality
\be\label{Poin}
\int_\Omega u^2\,\mathrm{d}x \le C_P \left(\int_{\Omega} |\nabla u|^2\,\mathrm{d}x+ \int_{\pa \Omega} u^2\,\mathrm{d}\sigma\right),
\ee
valid for every $u \in H^1(\Omega)$, we have that
\[
J(u,\Omega) \ge \frac12\int_{\Omega} |\nabla u|^2\,\mathrm{d}x - \frac{\lambda_+}{p} C \left(\int_{\Omega} |\nabla u|^2\,\mathrm{d}x \right)^\frac{p}2 -  \frac{\lambda_-}{q} C \left(\int_{\Omega} |\nabla u|^2\,\mathrm{d}x \right)^\frac{q}2-C,
\]
for some constant $C>0$ depending on the data, but independent of $u$. Since $1 \le p<q<2$, this gives at once the coercivity of $J(\cdot\,,\Omega)$, and the fact that $\inf_{\mathcal{A}} J(\cdot\,,\Omega)>-\infty$. Therefore, the existence of a minimizer follows from the direct method of the calculus of variations (the weak lower semi-continuity of the functional is a standard consequence of the Sobolev embeddings).
\end{proof}

Now we turn to the $\gamma_p$-non-degeneracy of local minimizers (recall Definition \ref{def.nondeg}).

\begin{proposition}\label{thm: non-deg}
Let $u \in H^1_{\loc}(\Omega)$ be a local minimizer of \eqref{eq0} or \eqref{equation1}. Then $u$ is $\gamma_p$-non-degenerate.
\end{proposition}
\begin{proof}
The proof is inspired by \cite[Proposition 3.3]{PaWe}. Suppose by contradiction that there exists $x_0 \in \{u=0\}$ and $r_k \searrow 0^+$ such that
  \be\label{deg.ass}
  \lim_{k\to \infty} \frac{1}{r_k^{\gamma_p}}\sup_{B_{r_k}(x_0)}|u| =0.
\ee
Since $p<2$, there exists a non-negative function $w \in H^1_0(B_1)$ such that $J(w,B_1)< 0$. For every $k$, we define
  $$
  w_k(x)=\begin{cases}
  r_k^{\gamma_p}w\left(\frac{x-x_0}{r_k}\right) & \mbox{in }B_{r_k}(x_0)\\
  0 & \mbox{in }B_1\setminus B_{r_k}(x_0).
  \end{cases}
  $$
  It is plain that $\|w_k\|_{H^1(\Omega)} \to 0$ as $k \to \infty$. Therefore, for sufficiently large $k$ the function $u+w_k$ is an admissible competitor in the definition of local minimality, whence it follows that
  \be\label{loca.min}
  J(u,B_{r_k}(x_0))\leq J(u+w_k,B_{r_k}(x_0)).
  \ee
  Now,
  \begin{align*}
  J(u+w_k,B_{r_k}(x_0)) =&\, J(u,B_{r_k}(x_0)) + J(w_k,B_{r_k}(x_0))\\
  &+\int_{B_{r_k}(x_0)}\nabla u\cdot \nabla w_k \, \mathrm{d}x +\lambda_-\int_{B_{r_k}(x_0)}\big( g(u) - g(u+w_k) \big) \, \mathrm{d}x\\
  &   -\lambda_+\int_{B_{r_k}(x_0)}\big( f(u+w_k)-f(u)-f(w_k) \big) \, \mathrm{d}x,
  \end{align*}
  where $f(t):=(t^+)^p/p$ and $g(t) := (t^-)^q/q$. By the convexity of $f$, we have that
  $$
  \int_{B_{r_k}(x_0)} \big( f(u+w_k)-f(w_k)\big) \, \mathrm{d}x \geq \int_{B_{r_k}(x_0)} (w_k^+)^{p-1} u \, \mathrm{d}x
  $$
(for $p=1$, this inequality follows directly from the definition of positive part) and hence, using also the positivity of $g$, we deduce that
  \be\label{07041}
  \begin{split}
J(u+w_k,B_{r_k}(x_0)) \le &\, J(u,B_{r_k}(x_0)) + J(w_k,B_{r_k}(x_0))\\
  &+\int_{B_{r_k}(x_0)}\nabla u\cdot \nabla w_k \, \mathrm{d}x +\frac{\lambda_-}{q}\int_{B_{r_k}(x_0)}(u^-)^q\,\mathrm{d}x \\
  & + \frac{\lambda_+}{p}\int_{B_{r_k}(x_0)}\left((u^+)^{p}-(w_k^+)^{p-1} u\right) \, \mathrm{d}x.
  \end{split}
  \ee
At this point we observe that $J(w,B_{r_k}(x_0)) = r_k^{n+p\gamma_p}J(w,B_1)$, by definition of $w_k$, while, by using the degeneracy assumption \eqref{deg.ass}, we have that
\begin{align*}
  \left|\int_{B_{r_k}(x_0)}\nabla u\cdot \nabla w_k \, \mathrm{d}x \right| &\leq \int_{B_{r_k}(x_0)}  \left| \lambda_+ (u^+)^{p-1} w_k - \lambda_-(u^-)^{q-1}w_k \right|\,\mathrm{d}x\\
  & \le C \left( \sup_{B_{r_k}(x_0)} |u|^{p-1} + \sup_{B_{r_k}(x_0)} |u|^{q-1}\right) r_k^{n+\gamma_p} \int_{B_1}|w|\, \mathrm{d}x \\
  & = C \left( o(r_k^{\gamma_p(p-1)})  + o(r_k^{\gamma_p(q-1)})  \right) r_k^{n + \gamma_p} = o(r_k^{n+ \gamma_p p})
\end{align*}
as $k \to \infty$, where we used the fact that $p<q$. Similarly,
\[
\int_{B_{r_k}(x_0)} (u^-)^q \,\mathrm{d}x \le o(r_k^{q\gamma_p}) r_k^n = o (r_k^{n+ p\gamma_p}),
\]
and
  $$
    \left|\int_{B_{r_k}(x_0)}\left((u^+)^{p}-(w_k^+)^{p-1} u\right)\,\mathrm{d}x \right| \le o(r_k^{p\gamma_p}) r_k^n + o(r_k^{\gamma_p})\int_{B_{r_k}(x_0)}(w_k^+)^{p-1}\mathrm{d}x = o(r_k^{n+p\gamma_p}),
  $$
  as $k \to \infty$.
  Finally, by collecting the previous estimates in \eqref{07041}, coming back to \eqref{loca.min}, and recalling that $J(w,B_1)<0$, we infer that
  $$
  J(u, B_{r_k}(x_0)) \leq J(u, B_{r_k}(x_0)) + r_k^{n+\gamma_p p} J(w,B_1) + o(r_k^{n + \gamma_p p})
  $$
as $k \to \infty$, which gives a contradiction for $k$ sufficiently large.
\end{proof}

\section{Blow-up analysis}\label{section.blowup}
In this Section we develop the blow-up analysis of the nodal set close to those points of $\{u=0\}$ such that $\mathcal{V}(u,x_0)\geq \gamma_p$. For the sake of clarity, we split the proof in two different subsection, each of one treating a different case in Theorem \ref{thm.mainblow}. 

Before proceeding with the proof of Theorem \ref{thm.mainblow}, we address the issue of the upper semi-continuity of the vanishing order with respect to the point $x_0$. This will be particularly useful in the second case of the theorem.

\begin{lemma}\label{lem: upper sc}
Let $1 \le p<q <2$, $\lambda_{1,k}>0$, $\lambda_{2,k} \ge 0$, with $\lambda_{i,k} \to \lambda_i \ge 0$ as $k \to \infty$, for $i=1,2$. Let $v_k \in H^1_{\loc}(B_3)$ be such that
\begin{equation}\label{sc.eq}
-\Delta v_k = \lambda_{1,k} (v_k^+)^{p-1} - \lambda_{2,k} (v_k^-)^{q-1} \quad \text{in $B_3$}.
\ee
Suppose that $x_k \in Z(v_k) \cap B_1$ with $\mathcal{V}(v_k,x_k) \ge \gamma_p$ for every $k$, and that $x_k \to \xi$ and $v_k \rightharpoonup \varphi$ weakly in $H^1_{\loc}(B_3)$, as $k \to +\infty$. Then
\[
\mathcal{V}(\varphi, \xi) \ge \gamma_p.
\]
\end{lemma}

\begin{proof}
This result can be proved by adapting the proof of \cite[Proposition 5.1]{soavesublinear} in the present setting (this requires minor changes). However, we present a much simpler proof here (valid also for solutions to \eqref{eq.sublinear.ST}).

\medskip

Suppose at first that $p=1$, when $\gamma_p = 2$. By weak $H^1_{\loc}$-convergence of $\{v_k\}$ and elliptic regularity for equation \eqref{sc.eq}, we have that $v_k \to \varphi$ in $C^{1,\alpha}(B_3)$, for every $\alpha \in (0,1)$. In particular, since $v_k(x_k) = 0 = |\nabla v_k(x_k)|$ (since $\mathcal{V}(v_k,x_k) \ge 2$), this implies that also $\varphi(\xi) = 0 = |\nabla \varphi(\xi)|$.

Now, if at least one between $\lambda_1$ and $\lambda_2$ are positive, Proposition \ref{prop.underdog} implies that $\mathcal{V}(\varphi,\xi) \in \{1\} \cup [2,+\infty)$. However, since both $\varphi$ and its gradient vanish in $\xi$, the case $\mathcal{V}(\varphi,\xi) =1$ is not possible, and hence the thesis follows.

If instead $\lambda_1=0=\lambda_2$, then $\varphi$ is a harmonic function vanishing in $\xi$ together with its gradient. Since the admissible vanishing orders of harmonic functions are positive integers, this directly implies that $\mathcal{V}(\varphi, \xi) \ge 2$, as desired.

\medskip

The case $p>1$ is more involved. Recall from \cite[Lemma 2.2]{soavesublinear} that, for any $k \in \N$, it is possible to choose $\delta_k \in \left[0,\frac{1}{2^k}\right)$ such that the sequence
\begin{equation}\label{def beta k delta}
\begin{cases}
\beta_{1}:= (p+1) \\
\beta_{k}:= (p-1) \beta_{k-1} + 2 -\delta_k
\end{cases}
\end{equation}
satisfies
\begin{equation}\label{lim beta_k}
\beta_{k} \not \in \N \quad \text{for every $k$, and} \quad \beta_{k} \nearrow \gamma_p \quad \text{as $k \to \infty$}.
\end{equation}
Observe also that by weak $H^1_{\loc}$-convergence of $\{v_k\}$ and elliptic regularity for equation \eqref{sc.eq}, we have that $v_k \to \varphi$ in $C^{2,\alpha}(B_3)$, for some $\alpha \in (0,1)$. Now the idea is to use the same argument in \cite[Proposition 2.1]{soavesublinear} uniformly along the sequence $\{v_k\}$.

Since $\mathcal{V}(v_k,x_k) \ge \gamma_p>1$, the sequences $\{\lambda_{i,k}\}$ ($i=1,2$) are bounded, and $\{v_k\}$ is bounded in $C^{2,\alpha}(B_{2})$, there exists $C>0$ independent of $k$ such that
\[
\begin{split}
|\Delta v_k(x)| &\le \sup\{\lambda_{1,k}\} |v_k(x)|^{p-1} +\sup\{\lambda_{2,k}\} |v_k(x)|^{q-1} \\
& \le \sup\{\lambda_{1,k}\}  C  |x-x_k|^{p-1} +\sup\{\lambda_{2,k}\} C |x-x_k|^{q-1} \\
& \le \max\left\{ C, 2^{p-1}\right\} |x-x_k|^{p-1}
\end{split}
\]
for every $x \in B_1(x_k)$ (where we used the fact that $p<q$). Thus, by \cite[Lemma 1.1]{CafFri85} there exist harmonic polynomials $P_{1,k}$ of degree $\lfloor p-1 \rfloor+2 =2$, and functions $\Gamma_{1,k}$ such that
\[
v_k(x) = P_{1,k}(x-x_k) + \Gamma_{1,k}(x) \qquad \text{in $B_1(x_k)$},
\]
with
\[
|\Gamma_{1,k}(x) | \le C_1  |x-x_k|^{p+1} \quad \text{in $B_1(x_k)$},
\]
for a positive constant $C_1>0$ independent of $k$ (indeed, $C_1$ depends only on upper bounds on $p$, $\|v_k\|_{W^{1,\infty}(\partial B_1(x_k))}$, $\lambda_{i,k}$ and $n$, and these quantities are all uniformly estimated in $k$). Now, the case $P_{1,k} \not \equiv 0$ is not possible, since it would give $\mathcal{V}(v_k, x_k) \le 2<\gamma_p$. Then $v_k = \Gamma_{1,k}$ for every $k$, and hence by the above estimates
\[
\begin{split}
|\Delta v_k(x)| &\le \sup\{\lambda_{1,k}\} |v_k(x)|^{p-1} +\sup\{\lambda_{2,k}\} |v_k(x)|^{q-1} \\
& \le \sup\{\lambda_{1,k}\}  C_1  |x-x_k|^{(p+1)(p-1)} +\sup\{\lambda_{2,k}\} C_1 |x-x_k|^{(p+1)(q-1)} \\
& \le \max\left\{ C, 2^{(p-1)(p+1)-\delta_2}\right\} |x-x_k|^{(p-1)(p+1)-\delta_2}
\end{split}
\]
for every $x \in B_1(x_k)$, with $(p+1)(p-1) -\delta_2 = \beta_2-2 \not \in \N$. As a consequence, we can apply again \cite[Lemma 1.1]{CafFri85}: letting
\[
\alpha_2:= \lfloor (p-1)(p+1) - \delta_2 \rfloor+2,
\]
there exist harmonic polynomials $P_{2,k}$ of degree $\alpha_2$, functions $\Gamma_{2,k}$, and a constant $C_2>0$ independent of $k$ such that
\[
v_k(x) = P_{2,k}(x-x_k) + \Gamma_{2,k}(x) \qquad \text{in $B_1(x_k)$},
\]
and
\[
|\Gamma_{2,k}(x) | \le C_2  |x-x_k|^{\beta_2} \quad \text{in $B_1(x_k)$}.
\]
The case $P_{2,k} \not \equiv 0$ is possible only if $\beta_2 \ge \gamma_p$. If not, then $v_k=\Gamma_{2,k}$ and we can iterate the previous argument in the following way: for any $k \ge 3$ such that $P_{m-1,k} \equiv 0$, we let
\begin{equation}\label{def beta_k}
\alpha_m:= \lfloor (p-1)\beta_{m-1} - \delta_m \rfloor+2;
\end{equation}
then there exist harmonic polynomials $P_{m,k}$ of degree $\alpha_m$, functions $\Gamma_{m,k}$, and constants $C_m>0$ independent of $k$ such that
\[
v_k(x) = P_{m,k}(x-x_k) + \Gamma_{m,k}(x) \qquad \text{in $B_1(x_k)$},
\]
and
\[
|\Gamma_{m,k}(x) | \le C_m  |x-x_k|^{\beta_m} \quad \text{in $B_1(x_k)$}.
\]
Since $\beta_m \nearrow \gamma_p$, and $\mathcal{V}(v_k,x_k) \ge \gamma_p$, we deduce that for any fixed $\eps$ there exists $\bar m \in \N$ with $\gamma_p -\eps \le \beta_{\bar m}< \gamma_p$, and
\[
|v_k(x)| = |\Gamma_{\bar m,k}(x)|  \le C_{\bar m} |x-x_k|^{\gamma_p-\eps} \quad \text{in $B_1(x_k)$}
\]
By taking the limit as $k \to +\infty$, by local uniform convergence we deduce that
\[
|\varphi(x)| \le C_{\bar m} |x-\xi|^{\gamma_p-\eps} \quad \text{in $B_{1/2}(\xi)$}.
\]
Since $\eps>0$ can be arbitrarily chosen, this implies that $\mathcal{V}(\varphi, \xi) \ge \gamma_p$, directly from the definition of vanishing order.
\end{proof}

\subsection{Proof of Theorem \ref{thm.mainblow}-(i)}
In this subsection, we focus on the first alternative in  Theorem \ref{thm.mainblow}. Namely, we suppose that \eqref{hp blow-up reg} holds, and we prove the thesis in point (i) of the theorem.

At first, we introduce some notation. We define the (natural) blow-up family centered at $x_0$ as the family of scaled functions
    \be \label{blow.up}
    u_{x_0,r}(x)= \frac{1}{r^{\gamma_p}}u(x_0 + r x)\quad \mbox{for } x\in B_{x_0,r}=\frac{B_1 - x_0}{r},
    \ee
where $\gamma_p$ is defined in \eqref{gamma_p}; note that
\be\label{blow.up.equation}
-\Delta u_{x_0,r} = \lambda_+(u_{x_0,r}^+)^{p-1}- r^{(q-p)\gamma_p} \lambda_-(u_{x_0,r}^-)^{q-1}\quad\mbox{in }B_{x_0,r},
\ee
where we used the fact that $2-(2-q)\gamma_p=(q-p)\gamma_p$. Note that $r^{(q-p)\gamma_p} \to 0$ as $r \to 0^+$, since $p<q$.

In this setting, recalling also Remark \ref{rem: su W scaled}, the following identity holds for the Weiss-type functionals:
\be\label{blow.up.weiss}
\begin{split}
W_{\gamma_p,2}(u_{x_0,r},0,\rho)  = & \frac{1}{\rho^{n-2+2\gamma_p}} \int_{B_\rho} \left(|\nabla u_{x_0,r}|^2 - 2 F_{\lambda_+, r^{(q-p)\gamma_p} \lambda_-}(u_{x_0,r})\right)\,\mathrm{d}x \\
& - \frac{\gamma_p}{\rho^{n-1+2\gamma_p}} \int_{\pa B_\rho} u_{x_0,r}^2\,\mathrm{d}\sigma \\
= &  W_{\gamma_p,2}(u,x_0,\rho r).
 \end{split}
\ee

In Lemma \ref{lem: ex limit} (see also the discussion preceding it), we noticed that for solutions of both \eqref{eq0} and \eqref{equation1} there exists the limit $W_{\gamma_p,2}(u,x_0,0^+)$ at points $x_0$ with vanishing order $\mathcal{V}(u,x_0) \ge \gamma_p$. In the next lemma, we show that this limit can be $-\infty$ only if we are in case (ii) of Theorem \ref{thm.mainblow}.

\begin{lemma}\label{lem.1}
  Let $x_0 \in \{u=0\}$ be such that $\mathcal{V}(u,x_0)\geq \gamma_p$. If $W_{\gamma_p,2}(u,x_0,0^+)=-\infty$, then necessarily \eqref{case2} holds, namely
  \[
  \limsup_{r \to 0^+} \frac{H(u,x_0,r)}{r^{n-1+2\gamma_p}}=+\infty.
  \]
\end{lemma}
\begin{proof}
In terms of the scaled family $u_{x_0,r}$ defined in \eqref{blow.up}, we can rephrase \eqref{case2} as follows: there exists $r_k \to 0^+$ such that
  $$
  \lim_{k\to +\infty}\int_{\partial B_1}u_{x_0,r_k}^2 \,\mathrm{d}\sigma = +\infty.
  $$
  Thus, by contradiction, suppose that $W_{\gamma_p,2}(u,x_0,0^+) = -\infty$ and there exists $r_0,C>0$ such that
  \be\label{contr0804}
  \int_{\partial B_1}u_{x_0,r}^2\, \mathrm{d}\sigma \leq C, \quad\forall r < r_0:= 1-|x_0|.
  \ee
By \eqref{blow.up.weiss}, we deduce that
  \begin{align*}
  -\infty &= \lim_{r \to 0^+} W_{\gamma_p,2}(u_{x_0,r},0,1) \\
  & \ge -C\gamma_p - \limsup_{r \to 0^+} \left( \frac{2\lambda_+}{p} \int_{B_1} (u_{x_0,r}^+)^p\, \mathrm{d}x + \frac{2\lambda_- r^{(q-p)\gamma_p}}{q}\int_{B_1} (u_{x_0,r}^-)^q\, \mathrm{d}x \right).
  \end{align*}
By definition of $\mathcal{V}(u,x_0)$, for any $\beta<\mathcal{V}(u,x_0)$ there exists $C>0$ such that
\[
\frac{H(u,x_0,r)}{r^{n-1}} \le C r^{2\beta} \quad \forall r <r_0.
\]
As a consequence, by choosing $p \gamma_p/q <\beta<\mathcal{V}(u,x_0)$ (in such a way that $2-2\gamma_p+q\beta>0$; note that this is possible, since $\mathcal{V}(u,x_0) \ge \gamma_p>p \gamma_p/q $), we obtain
\[
\begin{split}
r^{(q-p)\gamma_p} \int_{B_1} (u_{x_0,r}^-)^q\,\mathrm{d}x &= \frac{1}{r^{n-2+2\gamma_p}} \int_{B_r(x_0)} (u^-)^q \, \mathrm{d}x  \\
& \le \frac{1}{r^{n-2+2\gamma_p}} \int_0^r \rho^{n-1} \left(\frac{1}{\rho^{n-1}} H(u,x_0,\rho)\right)^\frac{q}{2}\,\mathrm{d}\sigma\\
& \le \frac{C}{r^{n-2+2\gamma_p}} \int_0^r \rho^{n-1 + q \beta}\,\mathrm{d}\rho \le C r^{2-2\gamma_p+ q\beta} \to 0,
\end{split}
\]
which implies that necessarily
$$
\limsup_{r \to 0^+}
  \int_{B_1} (u_{x_0,r}^+)^p\,\mathrm{d}x = +\infty.
$$
Thus, by collecting the $L^p$-norm of $u^+_{x_0,r}$ in the definition of $W_{\gamma_p,2}(u,x_0,r)$, we also obtain
\begin{align*}
-\infty &=\lim_{r \to 0^+} W_{\gamma_p,2}(u,x_0,r) \\
&=\lim_{r \to 0^+}   \int_{B_1} (u_{x_0,r}^+)^p\,\mathrm{d}x \left(
  \frac{\int_{B_1}{\abs{\nabla u_{x_0,r}}^2\,\mathrm{d}x}}{  \int_{B_1} (u_{x_0,r}^+)^p\,\mathrm{d}x}-\frac2p\lambda_+  \right. \\
  & \left. \hphantom{=\lim_{r \to 0^+}   \int_{B_1} (u_{x_0,r}^+)^p\,\mathrm{d}x \Big( \quad} - \frac2q \lambda_- r^{(q-p)\gamma_p}  \frac{\int_{B_1} (u_{x_0,r}^-)^q\,\mathrm{d}x}{ \int_{B_1} (u_{x_0,r}^+)^p\,\mathrm{d}x}   -\gamma_p\frac{\int_{\partial B_1}{u_{x_0,r}^2 \,\mathrm{d}\sigma}}{  \int_{B_1} (u_{x_0,r}^+)^p\, \mathrm{d}x} \right)\\
&=\lim_{r \to 0^+}   \int_{B_1} (u_{x_0,r}^+)^p\,\mathrm{d}x \left(
  \frac{\int_{B_1}{\abs{\nabla u_{x_0,r}}^2\,\mathrm{d}x}}{  \int_{B_1} (u_{x_0,r}^+)^p\,\mathrm{d}x}-\frac2p\lambda_+ + o(1) \right),
\end{align*}
which implies that
\be\label{equation.altro}
\limsup_{r \to 0^+}\frac{\int_{B_1}{\abs{\nabla u_{x_0,r}}^2\,\mathrm{d}x}}{  \int_{B_1} (u_{x_0,r}^+)^p\,\mathrm{d}x} \in \left[0,\frac2p\lambda_+\right].
\ee
Therefore, if we consider the normalized sequence
$$
v_{r}(x)=\frac{u_{x_0,r}(x)}{\norm{u^+_{x_0,r}}{L^p(B_1)}},
$$
such that $\norm{v^+_r}{L^p(B_1)}=1$, by the Poincar\'e inequality \eqref{Poin} we have that
\begin{align*}
\lim_{r \to 0^+}\norm{v_r}{H^1(B_1)}^2
& \le C\lim_{r \to 0^+} \frac{1}{\norm{u^+_{x_0,r}}{L^p(B_1)}^{2}}\left(\int_{B_1}|\nabla u_{x_0,r}|^2 \, \mathrm{d}x + \int_{\partial B_1}u^2_{x_0,r}\, \mathrm{d}\sigma\right)\\
&= \lim_{r \to 0^+} \frac{1}{\norm{u^+_{x_0,r}}{L^p(B_1)}^{2-p}}\left(
\frac{\int_{B_1}{\abs{\nabla u_{x_0,r}}^2\,\mathrm{d}x}}{\int_{B_1} (u_{x_0,r}^+)^p\,\mathrm{d}x}  + \frac{\int_{\partial B_1}u^2_{x_0,r}\,\mathrm{d}\sigma}{\int_{B_1}(u_{x_0,r}^+)^p\,\mathrm{d}x} \right)=0
\end{align*}
where we used that $p<2$, and the uniform bounds in \eqref{contr0804} and \eqref{equation.altro}. Thus, on one hand $v_r \to 0$ strongly in $H^1(B_1)$. On the other hand, by definition of $v_r$ ,we have
$$
1= \norm{v_r^+}{L^p(B_1)} \leq |B_1|^{\frac{1}{p\gamma_p}}\norm{v_r}{L^2(B_1)},
$$
which gives a contradiction for $r$ sufficiently small.
\end{proof}

Now we analyze the degenerate case.

\begin{lemma}\label{lem: deg}
The following equivalence holds: there exists $r_k \to 0^+$ such that $u_{x_0,r_k} \to 0$ uniformly in $B_1$ if and only if $u$ is $\gamma_p$-degenerate at $x_0$, that is
\[
\liminf_{r \to 0^+} \frac{1}{r^{\gamma_p}} \sup_{B_r(x_0)} |u| = 0.
\]
\end{lemma}

\begin{proof}
The proof follows straightforwardly from the fact that
\[
\sup_{B_{1}}|u_{x_0,r_k}| = \frac{1}{r_k^{\gamma_p}}\sup_{B_{r_k}(x_0)}|u|. \qedhere
\]
\end{proof}

Now we proceed with the:

\begin{proof}[Proof of Theorem \ref{thm.mainblow}-(i)]
By Lemma \ref{lem.1}, we already know that necessarily $W_{\gamma_p,2}(u,x_0,0^+) >-\infty$, since \eqref{hp blow-up reg} holds.

Let $\rho>0$ be fixed. In the first part of the proof, we show that $\{u_{x_0,r}\}$ is bounded in $H^1(B_\rho)$. To this end, we prove several estimates where $C>0$ denotes a positive constant whose value may change from line to line, and which may depend on $\rho$, $p$, $q$, $\mathcal{V}(u,x_0)$, but is independent of $r \in (0, r_0/\rho)$, with $r_0 = 1-|x_0|$. We will not stress such dependence, for the sake of brevity.

By the almost monotonicity of the Weiss-type functional $W_{\gamma_p,2}$ in Lemma \ref{lem: ex limit}, there exists $C>0$ such that
   \be\label{eq.sopra}
   W_{\gamma_p,2}(u_{x_0,r},0,\rho) = W_{\gamma_p,2}(u,x_0,\rho r) \leq W_{\gamma_p,2}(u,x_0,r_0) + C \le C
   \ee
   for every $r \in (0, r_0/\rho)$.
   Now, by H\"older and Poincar\'e inequalities
   \be\label{stima term p}
   \begin{split}
   \frac{1}{\rho^n}\int_{B_\rho}(u_{x_0,r}^+)^p\,\mathrm{d}x &\leq
   C \left(\frac{1}{\rho^n}\int_{B_\rho}u_{x_0,r}^2\, \mathrm{d}x\right)^{p/2}\\ &\leq C\left(\frac{1}{\rho^{n-2}}\int_{B_\rho}|\nabla u_{x_0,r}|^2\, \mathrm{d}x + \frac{1}{\rho^{n-1}} \int_{\pa B_\rho}  u_{x_0,r}^2\, \mathrm{d}\sigma\right)^{p/2}\\
   &\leq C\left(\frac{1}{\rho^{n-2}}\int_{B_\rho}|\nabla u_{x_0,r}|^2\,\mathrm{d}x\right)^{p/2}+C,
   \end{split}
   \ee
 where we also used \eqref{hp blow-up reg}. Analogously,
\[
   \frac{1}{\rho^n}\int_{B_\rho}(u_{x_0,r}^-)^q\,\mathrm{d}x \le C\left(\frac{1}{\rho^{n-2}}\int_{B_\rho}|\nabla u_{x_0,r}|^2\,\mathrm{d}x\right)^{q/2}+C.
\]
Recalling the expression of $W_{\gamma_p,2}(u_{x_0,r},0,\rho)$ from \eqref{blow.up.weiss}, and using the above estimates and \eqref{hp blow-up reg} into \eqref{eq.sopra}, we infer that
\[
\int_{B_\rho}|\nabla u_{x_0,r}|^2\mathrm{d}x \le C+ \left(\int_{B_\rho}|\nabla u_{x_0,r}|^2\,\mathrm{d}x\right)^{p/2} + C r^{(q-p)\gamma_p} \left(\int_{B_\rho}|\nabla u_{x_0,r}|^2\,\mathrm{d}x\right)^{q/2},
\]
for every $r>0$ small. Since $1\leq p,q<2$, this estimate, together with \eqref{hp blow-up reg} gives the boundedness of $\{u_{x_0,r}\}$ in $H^1(B_\rho)$ for small $r$.

This argument can be applied to every $\rho>0$; thus, by a diagonal selection, there exists a subsequence $\{u_{x_0,r_k}\}$ converging weakly in $H^1_\loc(\R^n)$ to a limit $\bar u \in H^1_\loc(\R^n)$; by taking the limit in the equation of $u_{x_0,r}$, namely \eqref{blow.up.equation}, we deduce that $\bar u$ solves \eqref{equation1} on the whole space (notice in particular that the coefficient of the negative phase vanishes in the limit):
\[
-\Delta \bar u= \lambda_+ (\bar u^+)^{p-1} \quad \text{in }\R^n.
\]
By elliptic regularity, this gives $C^{1,\alpha}_{\loc}(\R^n)$ convergence for every $\alpha \in (0,1)$ (actually, if $p>1$ then we also have $C^{2,\alpha}_{\loc}$ convergence, for some $\alpha \in (0,1)$).

Now we show that $\overline{u}$ is $\gamma_p$-homogeneous. Recalling the existence of the limit $W_{\gamma_p,2}(u,x_0,0^+)$ (see Lemma \ref{lem: ex limit}), equality \eqref{blow.up.weiss}, and using the $C^{1,\alpha}_{\loc}$-convergence, we deduce that
\begin{align*}
W_{\gamma_p,2}(u,x_0,0^+) &=\lim_{k \to \infty}W_{\gamma_p,2}(u,x_0,\rho r_k)\\
&=
\lim_{k \to \infty}W_{\gamma_p,2}(u_{x_0,r_k},0,\rho ) = W_{\gamma_p,2}(\overline{u},0,\rho).
\end{align*}
Since this is valid for every $\rho>0$, Proposition \ref{weiss.mon} implies that $\overline{u}$ is $\gamma_p$-homogeneous in $\R^n$, as desired.

Finally, we observe that if $W_{\gamma_p,2}(u,x_0,0^+)<0$, then also $W_{\gamma_p,2}(\overline{u},0,\rho)<0$, and hence $\bar u \not \equiv 0$. In particular, since the limit is the same for every sequence $r_k \to 0^+$, the fact that $W_{\gamma_p,2}(u,x_0,0^+)<0$ implies that $\bar u \not \equiv 0$ for any blow-up limit.

On the contrary, if $W_{\gamma_p,2}(u,x_0,0^+)\geq 0$, equation \eqref{equation1} and the homogeneity of $\bar u$ give
$$
0\leq W_{\gamma_p,2}(u,x_0,0^+)
=
W_{\gamma_p,2}(\overline{u},0,1)
=\left(1-\frac2p \right)\lambda_+\int_{B_1}(\overline{u}^+)^p\,\mathrm{d}x \le 0,
$$
whence $\overline{u}^+\equiv 0$ in $B_1$, and hence also in $\R^n$ by homogeneity. Therefore, $\overline{u}$ is a non-positive harmonic function in $\R^n$, vanishing at $0$, and by the maximum principle it follows that $\overline{u}\equiv 0$, whence $W_{\gamma_p,2}(x_0,u,0^+)=0$. Again, this shows that $W_{\gamma_p,2}(x_0,u,0^+) \ge 0$ implies that any blow-up limit must vanish. Moreover, as shown in Lemma \ref{lem: deg}, this fact is equivalent to the $\gamma_p$-degeneracy of $u$ at $x_0$.
\end{proof}

\subsection{Proof of Theorem \ref{thm.mainblow}-(ii)}
Now we focus on the second alternative in Theorem \ref{thm.mainblow}. In this case it is convenient to introduce the normalized blow-up family centered at $x_0$ as
    \be\label{blow.up.normalized}
    \widetilde u_{x_0,r}(x)= \frac{1}{h_r}u(x_0 + r x), \quad \text{where} \quad h_r:= \sqrt{\frac{1}{r^{n-1}} H(u,x_0,r)},
    \ee
once again defined in $B_{x_0,r}$. Note that $\|\widetilde u_{x_0,r}\|_{L^2(\partial B_1)} = 1$ for every $r$, and $\widetilde u_{x_0,r}$ satisfies
\be\label{blow.up.normalized.equation}
-\Delta \widetilde u_{x_0,r} = \lambda_+   \frac{r^2}{h_r^{2-p}} (\widetilde u_{x_0,r}^+)^{p-1} - \lambda_-  \frac{r^2}{h_r^{2-q}} (\widetilde u_{x_0,r}^-)^{q-1} \quad\mbox{in }B_{x_0,r} = \frac{B_1 - x_0}{r},
\ee
where
\[
\frac{r^2}{h_r^{2-p}} = \left(\frac{r^{n-1+2\gamma_p}}{H(u,x_0,r)}\right)^{\frac{1}{\gamma_p}}, \quad \frac{r^2}{h_r^{2-q}} = \left(\frac{r^{n-1+2\gamma_q}}{H(u,x_0,r)}\right)^{\frac{1}{\gamma_q}}.
\]
In this setting, recalling also Remark \ref{rem: su W scaled}, we have
\be\label{weiss tilde}
\begin{split}
W_{\gamma_p,2}(\widetilde u_{x_0,r},0,\rho ) =&\, \frac{1}{\rho^{n-2+2\gamma_p}} \int_{B_\rho} \left(|\nabla \widetilde u_{x_0,r}|^2 - 2 F_{  \frac{r^2}{h_r^{2-p}} \lambda_+,  \frac{r^2}{h_r^{2-q}} \lambda_-}(\widetilde u_{x_0,r})\right)\,\mathrm{d}x \\
& - \frac{\gamma_p}{\rho^{n-1+2\gamma_p}} \int_{\pa B_\rho} \widetilde u_{x_0,r}^2\,\mathrm{d}\sigma \\
= & \frac{r^{n-1+2\gamma_p}}{ H(u,x_0,r)} W_{\gamma_p,2}(u,x_0,\rho r).
\end{split}
\ee

\begin{proof}[Proof of Theorem \ref{thm.mainblow}-(ii)]
By \eqref{case2}, there exists a sequence $r_k \to 0^+$ such that
\be\label{case2'}
\lim_{k \to \infty}\frac{1}{r_k^{n-1+2\gamma_p}}\int_{\partial B_{r_k}(x_0)}u^2\,\mathrm{d}\sigma = +\infty.
\ee
Note that the equality \eqref{weiss tilde} and the almost monotonicity in Lemma \ref{lem: ex limit} imply that
\be\label{weiss tilde 1}
\begin{split}
\int_{B_1} |\nabla \widetilde u_{x_0,r_k}|^2\,\mathrm{d}x & \le \frac{r_k^{n-1+2\gamma_p}}{ H(u,x_0,r_k)} C +\frac{2}{p} \left(\frac{r_k^{n-1+2\gamma_p}}{ H(u,x_0,r_k)}\right)^\frac1{\gamma_p} \lambda_+ \int_{B_1} (\widetilde u_{x_0,r_k}^+)^{p}\,\mathrm{d}x  \\ & \qquad  +\frac{2}{q} \left(\frac{r_k^{n-1+2\gamma_q}}{ H(u,x_0,r_k)}\right)^\frac1{\gamma_q} \lambda_- \int_{B_1} (\widetilde u_{x_0,r_k}^-)^{q}\,\mathrm{d}x + \gamma_p \int_{\pa B_1} \widetilde u_{x_0,r_k}^2\,\mathrm{d}\sigma.
\end{split}
\ee
By estimating the terms with the $L^p$ and $L^q$ norms of $\widetilde u_{x_0,r_k}^{\pm}$ exactly as in \eqref{stima term p}, and taking into account \eqref{case2'}, we deduce that
\[
\begin{split}
\int_{B_1} |\nabla \widetilde u_{x_0,r_k}|^2\,\mathrm{d}x  \le &\, o(1) + o(1) \left(\int_{B_1} |\nabla \tilde u_{x_0,r_k}|^2\,\mathrm{d}x\right)^\frac{p}{2} \\
&  \, +  o(1) \left(\int_{B_1} |\nabla \tilde u_{x_0,r_k}|^2\,\mathrm{d}x\right)^\frac{q}{2} + \gamma_p,
\end{split}
\]
which gives the boundedness of $\{\widetilde{u}_{x_0,r}\}$ in $H^1(B_1)$, since $1 \le p<q < 2$. Therefore, up to a further subsequence (still denoted by $\{r_k\}$) we have that $\widetilde u_{x_0,r_k} \rightharpoonup  \widetilde u$ weakly in $H^1(B_1)$ and strongly in $L^2(\partial B_1)$ by compactness of the trace operator, and moreover $\widetilde u \not \equiv 0$, since $\|\widetilde u\|_{L^2(\partial B_1)} =1$. By taking the limit in equation \eqref{blow.up.normalized.equation}, and using again \eqref{case2'}, we deduce that $\widetilde u$ is harmonic in $B_1$, and hence by elliptic regularity $\widetilde u_{x_0,r_k} \to  \widetilde u$ in $C^{1,\alpha}(B_1)$, for every $\alpha \in (0,1)$.

To prove the homogeneity of $\widetilde u$, we observe at first that Lemma \ref{lem: upper sc} ensures that $\mathcal{V}(\widetilde u,0) \ge \gamma_p$. Since $\widetilde u$ is harmonic, this implies that $D^\sigma \widetilde u(0) = 0$ whenever $|\sigma| \le \beta_p$ (recall that $\beta_p$ is the maximal positive integer strictly smaller than $\gamma_p$, see \eqref{betap}). Moreover, by taking the limit into \eqref{weiss tilde 1}, we deduce that
\[
\int_{B_1} |\nabla \widetilde u|^2\,\mathrm{d}x \le \gamma_p \int_{\pa B_1} \widetilde u^2\,\mathrm{d}\sigma,
\]
and hence the homogeneity of $\widetilde u$ follows from \cite[Lemma 4.2]{Wei2001}.
\end{proof}

\subsection{Classification of vanishing orders for local minimizer}

\begin{proof}[Proof of Theorem \ref{thm: van ord}]
Let $u$ be a local minimizer of \eqref{eq0} or \eqref{equation1} in $B_1$. If $x_0 \in \{u=0\}$, then either Theorem \ref{thm.mainblow}-(ii) holds, and in this case $\mathcal{V}(u,x_0) = \gamma_p$ by definition of $\mathcal{V}(u,x_0)$, or Theorem \ref{thm.mainblow}-(i) holds. If this latter alternative occurs, then up to a subsequence $u_{x_0,r} \to \bar u$, where $\bar u$ is a $\gamma_p$-homogeneous solution to \eqref{equation1}. Moreover, by Proposition \ref{thm: non-deg}, we deduce that $\bar u \not \equiv 0$ and this, by definition of $u_{x_0,r}$, implies that
\[
\frac{H(u,x_0,r)}{r^{n-1+2\gamma_p}} = H(u_{x_0,r},0,1) \to H(\bar u,0,1) \in (0,+\infty).
\]
By definition of vanishing order, this is again possible only if $\mathcal{V}(u,x_0) = \gamma_p$.
\end{proof}

\section{Homogeneous solutions in two dimensions and the Hausdorff \\dimension
 of the singular set}\label{section.two}

In this section we classify the global homogeneous solutions of \eqref{equation1} in the two-dimensional case (Theorem \ref{thm: homog}), and we derive the estimate for the Hausdorff dimension of the singular set for $\gamma_p$-non-degenerate solutions of \eqref{equation1} and \eqref{eq0} (Theorem \ref{thm: hausdorff}).

\medskip

We start with the proof of Theorem \ref{thm: homog}, which will follow from the combination of two results: at first, we prove existence of solutions with specific nodal properties (Proposition \ref{prop.existence}), and, afterwards, we show that there are no other solutions, modulo rotations (Proposition \ref{prop.uniqueness}).

It is plain that $u$ is a global homogeneous solution to \eqref{equation1} in $\R^2$ if and only if $u(r,\theta)= r^{\gamma_p}\varphi(\theta)$ and $\varphi$ is a solution to
\be\label{equation.h}
-\varphi'' -\gamma_p^2 \varphi = \lambda_+ (\varphi^+)^{p-1} \quad\text{on the unit circle $S^1$}
\ee
(namely a solution on $[0,2\pi]$, with periodic boundary conditions). We start by showing the existence of a solution to \eqref{equation.h} in a subinterval with a prescribed number of zeros.
\begin{lemma}\label{lem.Tk}
Let $k \in \N$ be such that $k \in (\gamma_p, 2\gamma_p)$. There exists a unique (non-trivial) solution to
\be\label{eq.homo.k}
\begin{cases}
  -\phi'' -\gamma_p^2 \phi = \lambda_+ (\phi^+)^{p-1}  & \mbox{in }(0,T_k) \\
  \phi(0)=0=\phi(T_k)
\end{cases}\quad\mbox{with}\quad T_k:=\frac{2\pi}{k},
\ee
which changes sign exactly once in $(0,T_k)$, being negative for $\theta$ close to $0$, and satisfies $\phi'(0^+)=-\phi'(T_k^-)$.
\end{lemma}
\begin{proof}
  Since $\phi^+$ and $\phi^-$ solve different ODEs in their supports, we need to distinguish between these two cases, and classify the possible openings of the sets $\{\phi^+>0\}$ and $\{\phi^->0\}$. For the sake of simplicity we divide the proof in three steps.

  {\it Step 1 - The negative part. }Let $\bar{\theta}\in (0,2\pi)$ and suppose that 
  \be\label{equation.negative.h}
  \begin{cases}
  -\phi'' =\gamma_p^2 \phi & \mbox{in }(0,\bar{\theta}) \\
  \phi <0 & \mbox{in }(0,\bar\theta)\\
  \phi(0)=0=\phi(\bar{\theta}).
\end{cases}
\ee
Then, necessarily,
$$
\bar \theta = \theta_p := \frac{\pi}{\gamma_p}\quad\mbox{and}\quad \phi(\theta)=-A\sin\left(\gamma_p\theta\right)
$$
for some $A>0$. Therefore, if $\phi$ changes sign in $(0,T_k)$ then
$$
\theta_p <T_k \quad \iff \quad k< 2\gamma_p. 
$$

{\it Step 2 - The positive part. } In the remaining part of $(0,T_k)$ we construct a positive solution to
\be\label{equation.positive.h}
\begin{cases}
  -\phi'' -\gamma_p^2 \phi = \lambda_+ \phi^{p-1}  & \mbox{in }(\theta_p,T_k) \\
  \phi >0 & \mbox{in }(\theta_p,T_k)\\
  \phi(\theta_p)=0=\phi(T_k)
\end{cases}
\ee
by addressing the associated minimization problem. Let
$$
J_{(\theta_p,T_k)}(\phi)=\int_{\theta_p}^{T_k}\left(\frac12 (\phi')^2-\frac{\gamma_p^2}{2}\phi^2-\frac{\lambda_+}{p} (\phi^+)^p\right)\, \mathrm{d}\theta
$$
be the associated functional defined in $H^1_0([\theta_p,T_k])$. By Sobolev embeddings, it is clear that $J$ is weakly lower semi-continuous and, by applying the Poincar\'{e} inequality on $(\theta_p, T_k)$, we obtain that
$$
J_{(\theta_p,T_k)}(\phi) \geq \frac12\left(1-\frac{\gamma_p^2}{\Lambda_1}\right)\int_{\theta_p}^{T_k}(\phi')^2
-\frac{\lambda_+}{p}\frac{(T_k-\theta_p)^{\frac{2-p}{2}}}{\Lambda_1^{p/2}}\left(\int_{\theta_p}^{T_k}(\phi')^2\right)^{\frac{p}{2}}
$$
where $\Lambda_1$ denotes the first eigenvalue of the Dirichlet-Laplacian on $(\theta_p,T_k)$, namely
$$
\Lambda_1=\Lambda_1(\theta_p,T_k)=\frac{\pi^2}{(T_k-\theta_p)^2}.
$$
Since by a direct computation (using $T_k= 2\pi/k$ and $\theta_p=\pi/\gamma_p$) we have
$$
\gamma_p^2 < \Lambda_1 \quad \iff \quad  \frac{1}{\gamma_p}>\frac{2}{k}-\frac{1}{\gamma_p} \quad \iff \quad k>\gamma_p,
$$
and the latter condition is in force, we deduce that $J$ is bounded
from below and coercive in $H^1_0([\theta_p,T_k])$. Thus, by the direct method of the calculus of variations there exists a minimizer $\phi$ which solves the first equation of \eqref{equation.positive.h} with the boundary conditions. Since $J_{(\theta_p,T_k)}(|\phi|) \le J_{(\theta_p,T_k)}(\phi)$, with strict inequality if $\phi<0$ on a subinterval, we have that any minimizer is non-negative; thus, the strict positivity follows by the maximum principle.

{\it Step 3 - Construction of a sign-changing solution. }Let $\phi \in H^1_0([0,T_k])$ be such that
$$
\phi(\theta)=
\begin{cases}
  -A\sin(\gamma_p\theta) & \mbox{if } [0,\theta_p] \\
  \phi_2(\theta) & \mbox{if } [\theta_p,T_k]
\end{cases},
$$
where $A>0$ and $\phi_2$ is the solution to \eqref{equation.positive.h} of Step 2. In order to show that $\phi$ satisfies \eqref{eq.homo.k}, by the previous steps we only need to check that it is of class $C^1$ in a neighborhood of $\theta=\theta_p$. This is true for the choice $A=\phi'_2(\theta_p^+)/\gamma_p>0$. It remains to show that $\phi'(0^+) = \phi'(T_k^-)$. To this end, we note that the Hamiltonian function
\be\label{hamiltonian}
H(\phi,\phi')=\frac12 (\phi')^2 + \frac{\gamma_p^2}{2}\phi^2 + \frac{\lambda_+}{p}(\phi^+)^p
\ee
is constant along solutions to the first equation of \eqref{eq.homo.k}. Since $\phi(\theta) \to 0$ as both $\theta \to 0^+$ and $\theta \to T_k^-$, we deduce that $|\phi'(0^+)|=|\phi'(T_k^-)|$. The sign of $\phi$ finally implies that $\phi'(0^+)=-\phi'(T_k^-)$.

{\it Step 4 - Uniqueness of the solution. } Suppose by contradiction that there are two solutions $\phi_1, \phi_2$ with the desired properties. Note that the Hamiltonian function \eqref{hamiltonian} is constant along solutions to \eqref{eq.homo.k}:
\be\label{conservation}
\frac12 (\varphi'(\theta))^2 + \frac{\gamma_p^2}{2}\varphi^2(\theta) + \frac{\lambda_+}{p}(\varphi^+(\theta))^p = h \qquad \forall \theta \in [0,T_k]
\ee
for some $h \geq 0$, and actually $h>0$ unless $\varphi \equiv 0$. Then $\phi_1$ and $\phi_2$ are negative on $(0,\theta_p)$, and solve
\be\label{uni'}
\begin{cases}
  -\varphi'' -\gamma_p^2 \varphi = \lambda_+(\varphi^+)^{p-1}  & \mbox{in }(\theta_p,T_k) \\
  \varphi >0 & \mbox{in }(\theta_p,T_k)\\
  \varphi(\theta_p)=0, \qquad \varphi'(\theta_p)=2h_i^{1/2},
\end{cases}
\ee
for some $h_i>0$. Although the sublinear term in the first equation of \eqref{uni'} is only H\"{o}lder continuous, by the shape of the Hamiltonian function \eqref{conservation} it is easy to check that \eqref{uni'} has exactly one solution $\varphi$ (any level curves of the Hamiltonian does not cross the origin in the phase plane, apart from the equilibrium point $(0,0)$ itself). This means that, being $\phi_1 \neq \phi_2$, necessarily $h_1 \neq h_2$, say $h_1<h_2$. On the other hand, the conservation of the Hamiltonian also gives
\be\label{time}
T_k-\theta_p = 2\int_0^{M_i}\frac{\mathrm{d}\varphi}{\sqrt{2h-\gamma_p^2 \varphi^2 -\frac{2\lambda_+}{p}\varphi^p}} = 2M_i\int_0^{1}\frac{\mathrm{d}t}{\sqrt{2h-\gamma_p^2 M_i^2 t^2 -\frac{2\lambda_+}{p}M_i^p t^p}},
\ee
where $M_i>0$ is the maximum value of $\phi_i$ in $(\theta_p,T_k)$, which satisfies
\be\label{M}
\frac{\gamma_p^2}{2}M_i^2 + \frac{\lambda_+}{p}M_i^p = h_i.
\ee
By combining \eqref{time} and \eqref{M}, we further deduce that
\be\label{equazione.sopra}
T_k-\theta_p = T(M_i):=2\int_0^{1}\frac{\mathrm{d}t}{\sqrt{\gamma_p^2 (1- t^2) +\frac{2\lambda_+}{p M_i^{2-p}} (1-t^p)}},
\ee
which implies that $T_k-\theta_p=T(M_i)$ is a continuous function, strictly increasing with respect to $M_i$. But, by the inverse function theorem, the maximum value $M_i$ can be in turn seen as a continuous functions $M_i=M_i(h_i)$, strictly increasing with respect to $h_i$. Thus, on one hand
\[
h_1<h_2 \quad \implies \quad M_1<M_2 \quad \implies \quad T(M_1)<T(M_2),
\]
but on the other hand $T(M_1) = T(M_2) = T_k-\theta_p$, since $\phi_1$ and $\phi_2$ both solve \eqref{equation.positive.h}. This contradiction shows that \eqref{equation.positive.h} has a unique solution and, in turn, the solution constructed at Step 3 is unique.
\end{proof}

\begin{remark}\label{rem: opening}
From the previous proof, we learned that any connected component of $\{u<0\}$, when $u$ is any non-trivial $\gamma_p$-homogeneous solution to \eqref{equation1}, is a cone of opening $\theta_p=\pi/\gamma_p$.

Note also that we could obtain a unique solution $\tilde \phi$ in $[0,T_k]$ with $\tilde \phi(\theta)>0$ for $\theta$ close to $0$. Thanks to the condition $\phi'(0) = -\phi'(T_k)$, and the fact that equation \eqref{equation.h} is invariant under translations, it is clear that $\tilde \phi(\theta) = \phi(\theta_p+\theta)$.
\end{remark}

By extending the previous solution by periodicity on the whole interval $[0,2\pi]$, we obtain the existence of a $\gamma_p$-homogeneous solution to \eqref{equation1} in $\R^2$.
\begin{proposition}\label{prop.existence}
  Let $k \in \N$ be such that $k \in (\gamma_p,2\gamma_p)$. Then, there exists a non-trivial solution to
  $$
-\varphi'' -\gamma_p^2 \varphi = \lambda_+(\varphi^+)^{p-1} \quad\text{in $S^1$}
$$
with $2k$ zeros in $S^1$, that is
$$
\varphi(i T_k)=0=\varphi(\theta_p + iT_k)
\quad\mbox{for }i=0,\dots,k.
$$
This solution is unique modulo rotations, in the class of solutions with minimal period $2\pi/k$. \end{proposition}
\begin{proof}
Let $k\in \N$, $T_k=2\pi/k$ and let $\phi \in H^1_0([0,T_k])$ be the function constructed in Lemma \ref{lem.Tk}, satisfying $\phi'(0^+)=\phi'(T_k^-)$. By setting
$$
\varphi(\theta)=
\begin{cases}
  \phi(\theta) & \mbox{in } [0,T_k] \\
  \phi(\theta-T_k) & \mbox{in } [T_k,2T_k]\\
  \dots\\
  \phi(\theta-(k-1)T_k) & \mbox{in } [2\pi - T_k,2\pi]
\end{cases}
$$
we obtain a solution on the whole interval $[0,2\pi]$, satisfying $2\pi$-periodic conditions at the ends. The uniqueness follows directly from the uniqueness in Lemma \ref{lem.Tk} (one could also use the function $\tilde \phi$ of Remark \ref{rem: opening} instead of $\phi$. But, since $\tilde \phi$ is a translation of $\phi$, in the end one would obtain the same solution).
\end{proof}

Now we show that any $\gamma_p$-homogeneous solutions to \eqref{equation1} in $\R^2$ must be one of those constructed in Proposition \ref{prop.existence}.

\begin{proposition}\label{prop.uniqueness}
Let $u$ be a non-trivial $\gamma_p$-homogeneous solution to \eqref{equation1}. Then $u(r,\theta) = r^{\gamma_p} \varphi(\theta)$, where $\varphi$ is the $(2\pi/k)$-periodic extension of the solution of Lemma \ref{lem.Tk} for some $k \in (\gamma_p,2\gamma_p)$.

\end{proposition}
\begin{proof}
Since any $\gamma_p$-homogeneous solution vanishes in $0$, by the maximum principle it follows trivially that $\varphi$ must changes sign on $S^1$.

Now, from the conservation of the Hamiltonian function \eqref{hamiltonian} on the interval $[0,2\pi]$, we deduce that there exists $h>0$ such that
\be\label{der at 0}
\varphi'(\theta)=\pm 2h^{1/2} \quad \text{ for every zero point $\theta \in \{\varphi = 0\}$},
\ee

Consider three consecutive zero points $0\leq \theta_0<\theta_1<\theta_2\leq 2\pi$ and suppose that $\{\varphi<0\}=(\theta_0,\theta_1)$ and $\{\varphi>0\}=(\theta_1,\theta_2)$. As observed in Remark \ref{rem: opening}, it is necessary that
$\theta_1-\theta_0=\theta_p$ with $\theta_p = \pi/\gamma_p$. Therefore, by \eqref{der at 0} we have
$$
\varphi(\theta)=-\frac{2h^{1/2}}{\gamma_p}\sin(\gamma_p(\theta-\theta_0)) \quad\mbox{in }[\theta_0,\theta_1],
$$
and this expression also holds, up to a translation, on any other interval $(\theta_i, \theta_{i+1})$ where $\varphi<0$ and $\varphi(\theta_i) = 0 =\varphi(\theta_{i+1})$.

Similarly, on $(\theta_1,\theta_2)$ we have
\be\label{uni}
\begin{cases}
  -\varphi'' -\gamma_p^2 \varphi = \lambda_+(\varphi^+)^{p-1}  & \mbox{in }(\theta_1,\theta_2) \\
  \varphi >0 & \mbox{in }(\theta_1,\theta_2)\\
  \varphi(\theta_1)=0,  \qquad  \varphi'(\theta_1)=2h^{1/2},
\end{cases}
\ee
and this Cauchy problem has a unique solution (see Step 4 in Lemma \ref{lem.Tk}). But then, by \eqref{der at 0} again, the function $\varphi$ satisfies \eqref{uni} in any interval $(\theta_{j},\theta_{j+1})$ where $\varphi>0$ and $\varphi(\theta_j) = 0 = \varphi(\theta_{j+1})$.

This means that $\varphi$ is both $2\pi$ and $(\theta_p + \theta_2-\theta_1)$-periodic, with $\theta_p + \theta_2-\theta_1 \le 2\pi$, whence necessarily
\be\label{period}
\theta_p + \theta_2-\theta_1 = \frac{2\pi}{k}=:T_k,\quad\mbox{for some }k\in\N,
\ee
and $\varphi$ changes sign precisely once in $[0,T_k]$.

Now, if $k \in (\gamma_p,\gamma_{2p})$, then by Proposition \ref{prop.existence} we know that $\varphi$ must coincide with a $2\pi/k$-periodic extension of a solution found in Lemma \ref{lem.Tk}.

On the other hand, from the condition $T_k>\theta_p$, we immediately deduce that in order to have a $T_k$ periodic solution it is necessary that $k<2\gamma_p$.

In order to conclude the proof, we show that another necessary condition for the existence is $k>\gamma_p$. Let $\varphi$ be a $T_k$-periodic solution to \eqref{equation.h}, and suppose that $\varphi>0$ in $(\theta_1,\theta_2)$. Arguing as in Step 4 of Lemma \ref{lem.Tk}, we obtain that
\be\label{equazione.sopra}
\theta_2-\theta_1 = T(M):=2\int_0^{1}\frac{\mathrm{d}t}{\sqrt{\gamma_p^2 (1- t^2) +\frac{2\lambda_+}{p M^{2-p}} (1-t^p)}},
\ee
where $M$ is the maximum value of $\varphi$ in $(\theta_1,\theta_2)$, characterized by
\be\label{M'}
\frac{\gamma_p^2}{2}M^2 + \frac{\lambda_+}{p}M^p = h.
\ee
This implies that $\theta_1-\theta_2=T(M)$ is a continuous function, strictly increasing with respect to $M$.
Since $M \in (0,+\infty)$ and
\begin{align*}
\lim_{M\to 0^+ }T(M) &= \lim_{M\to 0^+ } 2 M^{\frac{2-p}{2}}\int_0^{1}\frac{\mathrm{d}t}{\sqrt{\frac{2\lambda_+}{p} (1-t^p)}}=0,\\
\lim_{M\to +\infty }T(M) &= \frac{2}{\gamma_p}\int_0^1\frac{\mathrm{d}t}{\sqrt{1-t^2}} = \frac{\pi}{\gamma_p} =\theta_p,
\end{align*}
we obtain that $T(M) \in (0,\theta_p)$, for every $M>0$. Therefore, by \eqref{period},
$$
T_k - \theta_p = \theta_2-\theta_1=T(M) \in (0,\theta_p) \quad \iff \quad T_k<2\theta_p \quad \iff \quad k>\gamma_p,
$$
which excludes the presence of $T_k$-periodic solution for $k\leq \gamma_p$.
\end{proof}

\begin{proof}[Proof of Theorem \ref{thm: homog}]
The proof follows directly by combining Propositions \ref{prop.existence} and \ref{prop.uniqueness}.
\end{proof}

Once that the classification of the $2$-dimensional solutions to \eqref{equation1} is settled, thanks to the blow-up Theorem \ref{thm.mainblow} we can easily establish the partial regularity of the singular set for $\gamma_p$-non-degenerate solutions.

\begin{proof}[Proof of Theorem \ref{thm: hausdorff}]
Thanks to Theorems \ref{thm.mainblow} and \ref{thm: homog}, and observing that the singular set of any $\gamma_p$-homogeneous solution has Hausdorff dimension at most $n-2$, the thesis follows in a rather standard way from the Federer dimension reduction principle. We refer the reader to the proofs of \cite[Proposition 6.3]{MW} and \cite[Corollary 6.1]{MW} for the details (alternatively, it is also possible to adapt the proof of \cite[Theorem 1.7]{soavesublinear}).
\end{proof}

\section{Construction of degenerate solutions in two dimensions}\label{section.twodeg}

In what follows we construct $\gamma_p$-degenerate two-dimensional solutions for both equations  \eqref{eq0} and \eqref{equation1}. We focus on the case $p>1$. If $p=1$ in \eqref{equation1}, then the same result was proved in \cite[Section 4]{AW}. The case $p=1$ in \eqref{eq0} could be treated adapting the argument we are going to explain, but requires a regularization as in \cite{AW}, and we preferred to not insist on it for the sake of brevity.

Let $k \in \mathbb{N}$, $k \ge 2$, and let $\ell_0,\dots,\ell_{k-1}$ be lines through the origin, with $\ell_i$ forming an angle $i \pi/k$ with the positive $x_1$-axis. Let also $T_i: \R^2 \to \R^2$ be the reflection with respect to $\ell_i$, and let us introduce the space of symmetric functions
\[
\begin{split}
C^{m,\alpha}_{S_k}(\overline{B_1}) &:= \left\{u \in C^{m,\alpha}(\overline{B_1}): \ u(T_i x) = u(x) \quad \text{in $B_1$}, \ \forall i=0,\dots,k-1\right\} \\
H^1_{S_k}(\overline{B_1}) & := \left\{u \in H^{1}(\overline{B_1}): \ u(T_i x) = u(x) \quad \text{a.e. in $B_1$}, \ \forall i=0,\dots,k-1\right\}.
\end{split}
\]

\begin{lemma}\label{lem.anderson}
Let $\lambda_+>0$, $\lambda_- \ge 0$, and $1<p < q<2$. For each $g \in C^{2,1}_{S_k}(\overline{B_1})$ there exists $\kappa \in \R$ such that the problem
  $$
   \begin{cases}
   -\Delta u = \lambda_+(u^+)^{p-1}-\lambda_-(u^-)^{q-1} & \mbox{in }B_1\\
   u = g- \kappa &\mbox{on }\partial B_1
      \end{cases}
   $$
has a solution $u \in C^{2,\gamma}_{S_k}(\overline{B_1})$, for some $\gamma \in (0,1)$, such that $u(0) = 0$.
\end{lemma}

\begin{proof}
For every $f \in C^{0,\gamma}_{S_k}(\overline{B_1})$, it is classical that the problem
  \be\label{pb f}
   \begin{cases}
   -\Delta u = f & \mbox{in }B_1\\
   u = g &\mbox{on }\partial B_1
      \end{cases}
   \ee
has a solution $u \in C^{2,\gamma}_{S_k}(\overline{B_1})$; for instance, one can show the existence of a weak solution by minimizing the associated energy functional in the space of functions in $H^1_{S_k}(B_1)$ with prescribed trace. The existence of a minimizer can be shown as in the proof of Theorem \ref{thm: ex}. The fact that a minimizer in the space of symmetric function is a solution to \eqref{pb f} follows from the principle of symmetric criticality \cite{Pal}, since $f$ and $g$ are assumed to be symmetric as well. The regularity theory ensures that any weak solution is in fact a classical $C^{2,\gamma}$ solution, up to the boundary.

In particular, given any $v \in C^{0,\alpha}_{S_k}(\overline{B_1})$, we can consider $T(v)$ as the symmetric solution to \eqref{pb f} with
\[
f(x) = \lambda_+\big((v(x)-v(0))^+\big)^{p-1} - \lambda_-\big((v(x)-v(0))^-\big)^{q-1}\in C^{0,\gamma}_{S_k}(\overline{B_1}), \quad \text{for }\gamma= \alpha (p-1).
\]
Notice that $T(v) \in C^{2,\gamma}_{S_k}(\overline{B_1})$. However, by thinking at $T(v)$ as a function in $C^{0,\alpha}_{S_k}(\overline{B_1})$, we have an operator $T: C^{0,\alpha}_{S_k}(\overline{B_1}) \to C^{0,\alpha}_{S_k}(\overline{B_1})$ which is compact, by Schauder estimates (see e.g. \cite[Theorem 6.6]{GT}). At this point we aim at applying the Schauder fixed point theorem \cite[Theorem 11.3]{GT} and, to this end, we have to check that there exists a positive constant $M$ such that
\be\label{unif bounds}
\|u\|_{C^{0,\alpha}(\overline{B_1})} \le M \quad \text{for every $u \in C^{0,\alpha}_{S_k}(\overline{B_1})$ satisfying $u = \sigma T(u)$, with $\sigma \in [0,1]$}.
\ee
The equation $u=\sigma T(u)$ can be rewritten as
\be\label{pb fixed point}
\begin{cases}
-\Delta u = \sigma \lambda_+\big((u-u(0))^+\big)^{p-1} - \sigma \lambda_-\big((u-u(0))^-\big)^{q-1}& \text{in $B_1$} \\
u = \sigma g & \text{on $\pa B_1$.}
\end{cases}
\ee
Thus, the function $w:= u-\sigma g$ is in $H_0^1(B_1)$ and solves
\[
-\Delta w = \sigma \lambda_+\big((u-u(0))^+\big)^{p-1} - \sigma \lambda_-\big((u-u(0))^-\big)^{q-1} + \sigma \Delta g \quad \text{in $B_1$}.
\]
By the De Giorgi-Moser iteration for $-\Delta w = f \in L^t(B_1)$ with $w \in H_0^1(B_1)$, we deduce that for any $t>2$
\[
\begin{split}
\|u\|_{L^\infty(B_1)} &\le \|u-\sigma g\|_{L^\infty(B_1)} + \|g\|_{L^\infty(B_1)} \\
& \le  C \|\sigma \lambda_+\big((u-u(0))^+\big)^{p-1} - \sigma \lambda_-\big((u-u(0))^-\big)^{q-1} + \sigma \Delta g\|_{L^t(B_1)} + \|g\|_{L^\infty(B_1)} \\
& \le C \left( \| u\|_{L^\infty(B_1)}^{q-1} + \|g\|_{C^{2,1}(\overline{B_1})}\right).
\end{split}
\]
In turn, since $q-1<1$, this implies that $\|u\|_{L^\infty(B_1)} \le C_0$, with $C_0$ independent of $\sigma$; coming back to \eqref{pb fixed point}, this gives the desired uniform bound \eqref{unif bounds}, via elliptic estimates.

At this point the existence of a fixed point $\bar u \in C^{0,\alpha}_{S_k}(\overline{B_1})$ for $T$ follows from the Schauder fixed point theorem \cite[Theorem 11.3]{GT}. The fixed point $\bar u$ is a solution to \eqref{pb fixed point} with $\sigma=1$ and hence, by regularity theory, it is of class $C^{2,\gamma}$ for some $\gamma>0$. By taking $u= \bar u-\bar u(0)$, and letting $\kappa = \bar u(0)$, the proof is complete.
\end{proof}

\begin{proposition}
  There exists a non-trivial solution $u \in H^1(B_1)$ to \eqref{eq0} or \eqref{equation1} such that $0 \in \{u=0\}$ and
  $$
  \liminf_{r \to 0^+} \frac{1}{r^{\gamma_p}}\sup_{B_r(x_0)}|u| =0,
  $$
  that is $u$ is $\gamma_p$-degenerate at $0$.
\end{proposition}
\begin{proof}
Under the notations of Lemma \ref{lem.anderson}, let $k \in \N$ be such that $k>2\gamma_p$, and consider
  $$
  g(r\cos\theta, r \sin \theta)= r^k\cos(k\theta).
  $$
 By Lemma \ref{lem.anderson}, there exist a non-trivial $u \in C^{2,\gamma}_{S_k}(\overline{B_1})$ and $\kappa \in \R$ such that
  $$
   \begin{cases}
   -\Delta u = \lambda_+(u^+)^{p-1} -\lambda_-(u^-)^{q-1} & \mbox{in }B_1\\
   u = g- \kappa &\mbox{on }\partial B_1,
   \end{cases} \quad \text{and} \quad u(0) = 0.
   $$
 The idea is to use now Theorem \ref{thm.mainblow} and Proposition \ref{prop.underdog} at $0$, to show that $u$ is $\gamma_p$-degenerate at $0$. We divide the proof in three cases.

   {\it Case 1.} If $\mathcal{V}(u,0)<\gamma_p$, by Proposition \ref{prop.underdog} there exists a non-trivial homogeneous harmonic polynomial $P_{0}$ of degree $d=\mathcal{V}(u,0)$ such that
        $$
        u(x) = P_{0}(x) + O(|x|^{\mathcal{V}(u,0)+\delta}) \quad \mbox{in }B_{1/2},
        $$
        for some $\delta>0$ sufficiently small. However, by construction, the polynomial must inherit the symmetries of $u$ (i.e. $P_0 \in C^{\infty}_{S_k}$), which implies that it is homogeneous of degree $d=k>2\gamma_p$, in contradiction with the assumption $d=\mathcal{V}(u,0)<\gamma_p$;

   {\it Case 2.} If $\mathcal{V}(u,0)\geq \gamma_p$ and
   $$
\limsup_{r \to 0^+}\frac{H(u,0,r)}{r^{n-1+2\gamma_p}} = +\infty
    $$
then, by Theorem \ref{thm.mainblow}-(ii), there exists a subsequence $r_k \searrow 0^+$ such that
$$
\widetilde u_{0,r_k}
\to \widetilde u\quad\mbox{in } C^{1,\alpha}_\loc(B_1),
$$
for every $\alpha \in (0,1)$, where $\widetilde u$ is a $\gamma_p$-homogeneous harmonic polynomial. Therefore, the contradiction follows the same path of {\it Case 1}.

{\it Case 3.} If $\mathcal{V}(u,0)\geq \gamma_p$ and
$$
  \limsup_{r \to 0^+}\frac{H(u,0,r)}{r^{2\gamma_p}} < +\infty,
$$
then, by Theorem \ref{thm.mainblow}-(i), there exists a sequence $r_k \searrow 0^+$ such that
$$
u_{0,r_k}
\to \overline{u} \quad\mbox{in } C^{1,\alpha}_\loc(\R^2),
$$
for every $\alpha \in (0,1)$, where $\overline{u}$ is a $\gamma_p$-homogeneous solution to \eqref{equation1} in $\R^2$, still symmetric with respect to any reflection $T_i$. If $\bar u \not \equiv 0$, by homogeneity and symmetry, any connected component of $\{\bar u<0\}$ must be a cone of opening smaller than $2\pi/k$. Moreover, $2\pi/k<\theta_p$, since $k>2\gamma_p$. As observed in Remark \ref{rem: opening}, this is however not possible. We conclude that necessarily $\bar u\equiv 0$ which, as stated in Theorem \ref{thm.mainblow}, is equivalent to the $\gamma_p$-degeneracy of $u$ at $0$.
\end{proof}

\begin{remark}
It is interesting that the previous proof fails for solutions to \eqref{eq.sublinear.ST} ($1 \le p=q<2$, with $\lambda_\pm>0$). In such case, there are infinitely many $\gamma_p$-homogeneous solutions, and the connected components of the positive and the negative parts can be arbitrarily small. We refer to \cite[Section 8]{soavesublinear}. In fact, it is also known that all solutions to \eqref{eq.sublinear.ST} are non-degenerate \cite[Theorem 1.4]{soavesublinear}. Thus, the degeneracy is strictly connected to the asymmetric behavior of the positive and the negative parts in the right hand side of \eqref{eq0} and \eqref{equation1}.
\end{remark}

\end{document}